\documentclass{amsart}

\usepackage{amsmath,amssymb}
\usepackage{enumerate}
\usepackage{amsthm}

\usepackage{amssymb}
\usepackage{indentfirst}
\usepackage{comment}
\usepackage{amsthm}
\usepackage{amsmath}
\usepackage[final]{graphicx}
\usepackage{amsmath,amssymb}
\usepackage{float}
\usepackage{txfonts}
\usepackage{mathtools}
\usepackage{stmaryrd}   

\usepackage{tikz}
\usetikzlibrary{cd}
\usetikzlibrary{quotes,angles,positioning,topaths,calc}

\title{Determinant formulas for Ihara zeta functions via simple cycles}
\author{Kosei Watanabe}
\address{Graduate School of Mathematics, Nagoya University, Chikusa-ku, Furo-cho, Nagoya, 464-8602,  Japan}
\email{watanabe.kosei.g8@s.mail.nagoya-u.ac.jp}

\subjclass[2020]{}

\date{June 6, 2026}

\newcommand{\HH}{\mathcal{H}}

\newtheorem{thm}{Theorem}[section]
{\theoremstyle{definition}\newtheorem{Def}[thm]{Definition}}
\newtheorem{prop}[thm]{Proposition}
\newtheorem{lem}[thm]{Lemma}
\newtheorem{cor}[thm]{Corollary}

\theoremstyle{definition}

\newtheorem{rem}[thm]{Remark}

\newtheorem{ex}[thm]{Example}

\newtheorem{set}[thm]{Setup}

\numberwithin{equation}{section}

\usepackage{latexsym}

\newcommand{\NN}{{\mathbb{N}}}
\newcommand{\ZZ}{{\mathbb{Z}}}

\makeatletter
\newcommand{\llangle}{\langle\!\langle}
\newcommand{\rrangle}{\rangle\!\rangle}
\makeatother

\begin{document}
\bibliographystyle{amsalpha+}
\begin{abstract}
Building upon the algebraic framework of trace monoids introduced by Giscard and Rochet, we establish a new determinant formula for the Ihara zeta functions of certain digraphs. We present our results through two main theorems. First, at the algebraic level, we show a general determinant formula expressed by a Cayley determinant over the hike monoid ring, which provides a unifying perspective on several known determinant expressions. Second, by evaluating the lengths of hikes through a natural ring homomorphism, we prove that the reciprocal of the Ihara zeta function can be explicitly expressed as a determinant whose dimension is equal to the number of simple cycles of the line digraph without backtrack.
\end{abstract}


\maketitle
\tableofcontents

\section{Introduction}
Ihara zeta functions were initially introduced
in \cite{I66} as a discrete analogue of the Selberg zeta function. 
Subsequently, Serre stated that Ihara zeta functions are 
interpreted as zeta functions of graphs (cf. \cite{Serre1, Serre2}).
Following these developments, Bass \cite{Bass}, 
Hashimoto \cite{H89}, and Sunada \cite{KS00}
reformulated Ihara zeta functions using terminologies in graph theory.

The Ihara zeta function of a graph is defined 
as an infinite product in general over its prime and reduced cycles. 
A specific property of these functions 
is that their reciprocals can be represented 
by two distinct determinant formulae,
each offering a different perspective.
The first, known as the Ihara expression (cf. \cite{Bass, I66}), 
involves a matrix whose dimension is equal to the number of vertices.
In contrast, the second, known as the Hashimoto expression (cf. \cite{H89}), 
utilizes a matrix with a size equal to twice the number of edges.

The Ihara zeta function of a digraph (also known as a directed graph) 
is defined in an analogous manner (cf. \cite{KS00,MS03}).
Similarly, it possesses both Ihara-type and Hashimoto-type determinant 
expressions. Indeed, an undirected graph can be naturally identified with a 
symmetric digraph by replacing each edge with a pair of oppositely directed 
edges. Under this identification, the Ihara and Hashimoto expressions for the 
symmetric digraph coincide with the formulas for the original undirected graph. 
In this sense, the digraph setting provides a consistent generalization of the undirected graph setting.

However, these classical determinant expressions depend directly on the number of vertices or edges. For example, consider the subdivision of a directed edge that has no inverse. This operation inserts a new vertex along the edge, which increases both the number of vertices and edges. While it preserves the essential cycle structure and does not create any new prime cycles under the backtrackless condition, it increases the matrix sizes in both the Ihara and Hashimoto expressions. This indicates that these traditional expressions do not solely depend on the intrinsic cycle structure of the graph, but are heavily affected by local properties such as the number of vertices and edges.

In contrast, an approach based on trace monoids of cycles shows a different behavior. Since the generators of the trace monoid correspond to simple cycles, their number is completely invariant under the subdivision of an edge without an inverse. This suggests that the trace monoid framework captures the intrinsic cycle structure of the graph. Crucially, our approach is completely unaffected by such local modifications.

To achieve this cycle-intrinsic formulation,
we first establish our framework at the algebraic level
of the hike monoid ring. Building on the work of 
Giscard and Rochet \cite{GR17}, we treat the Ihara zeta function
as the Hilbert series of a specific trace monoid,
whose elements are referred to as hikes.
This algebraic formulation is our first major result (Theorem \ref{thm:cliquehike}), 
where we show a formula expressed by a Cayley determinant
over the hike monoid ring. This formulation gives
an Ihara-zeta-theoretic interpretation
of existing Cayley determinant expressions.
For instance, the Cayley determinant expression
established by Choffrut and Goldwurm \cite[Proposition 2]{CG99}
under the restriction of a "transitive orientation" can be naturally 
reinterpreted within our framework as 
the identity $\mathrm{Cdet}(I-M) = p_\mathcal{H}(G(\mathcal{H}))$.

By evaluating the lengths of hikes in this algebraic framework
through a natural ring homomorphism, we successfully derive
our second major result (Theorem \ref{thm:main1}),
which directly provides a determinant formula
for the reciprocal of the Ihara zeta function itself.
The core idea is to focus on the complete set of simple cycles
of the line digraph without backtrack $L_{WB}(G)$.
Although the structure of independent cycles does not automatically
yield a simple determinant,
we overcome this difficulty by applying a specific "closure operation"
to the noncommutation graph $\tilde{G}(\mathcal{H})$,
embedding it into a digraph $\tilde{G}_+(\mathcal{H})$.
It should be emphasized, however, that such a closure cannot always be constructed for an arbitrary digraph; rather, our main determinant formula specifically holds for the class of digraphs whose noncommutation graphs admit a consistent acyclic closure.

Although Theorem \ref{thm:main1} can logically be viewed as a corollary 
of Theorem \ref{thm:cliquehike}, we present it as a standalone main theorem 
from the viewpoint of the Ihara zeta function.
\begin{thm}\label{thm:main1}
Let $G$ be a finite, simple, and connected digraph, and let $\zeta(G, u)$ be its Ihara zeta function. Under the following notation, the reciprocal of the Ihara zeta function satisfies the following determinant formula:
\begin{itemize}
    \item $L_{WB}(G)$: the line digraph without backtracks of $G$ (cf. Definition \ref{def:LWBG}),
    \item $\Sigma$: the set of all simple cycles in $L_{WB}(G)$ with $\# \Sigma=n$ (cf. Definition \ref{def:simplecycle}),
    \item $\mathcal{H}=\Sigma^*/I_{\mathcal{H}}$: the hike monoid (cf. Definition \ref{def:hikemonoid}),
    \item $\Tilde{G}(\mathcal{H})$: the noncommutation graph of $\mathcal{H}$ (cf. Definition \ref{def:noncomgraph}),
    \item $\Tilde{G}_+(\HH)$: a closure of $\Tilde{G}(\HH)$ (cf. Definition \ref{def:closure}),
    \item $l(h)$: the total length of a hike $h = h_{i_1} \dots h_{i_k}$ on $L_{WB}(G)$, defined by $l(h)=\sum_{j=1}^k l(h_{i_j})$,
    \item $D_w(\Tilde{G}_+(\HH))$: the diagonal matrix $\mathrm{diag}[u^{l(h_1)},\dots,u^{l(h_n)}]$,
    \item $A(\Tilde{G}_+(\HH))$: the adjacency matrix of $\Tilde{G}_+(\HH)$.
\end{itemize}
\begin{equation*}
\zeta(G,u)^{-1} = \det(I_n - D_w(\Tilde{G}_+(\HH))A(\Tilde{G}_+(\HH))).
\end{equation*}
\end{thm}

The remainder of this paper is organized as follows.
In Section $2$,
we review the basic definitions of digraphs 
and line digraphs without backtrack.
Section $3$ provides the necessary background on trace monoids.
In Section $4$, we define the trace monoid structure 
on Ihara zeta functions. 
Section $5$ presents our Main Theorem 
along with its rigorous proof and 
a detailed concrete example to illustrate the entire procedure.

\section{Graphs and digraphs}
In this section, we provide the fundamental definitions of graphs and digraphs. Subsequently, we define the Ihara zeta function within the framework of directed graphs as in \cite{KS00}.

\begin{Def}
A \textit{graph} $G=(V(G),E(G))$ consists of 
a set $V(G)$ of \textit{vertices} and 
a multiset $E(G)$ of \textit{edges}. 
The structure is characterized as follows:
\begin{itemize}
\item Each edge is an unordered pair $\{u,v\}$ for $u,v \in V(G)$.
An edge with $u=v$ is called a \textit{loop}.
\item The graph $G$
is \textit{finite} if both $V(G)$ and $E(G)$ are finite sets.
\item The graph $G$ has \textit{multiple edges} if $E(G)$
contains multiple copies of the same pair.
\item The graph $G$ is called a \textit{simple graph} 
if it has neither loops nor multiple edges.
\item The graph $G$ is \textit{connected} if for any two distinct vertices 
$u,v \in V(G)$, there exists a sequence $(v_0,v_1,\dots,v_n)$ of $V(G)$ 
such that $v_0=u, \ v_n=v$ and 
for each $i \in \{0,1,\dots,n-1\}$, $\{v_i,v_{i+1}\} \in E(G)$.
\end{itemize}
\end{Def}

Following the definition of undirected graphs,
we introduce the notation for directed graphs.

\begin{Def}
A \textit{digraph} (a.k.a. directed graph)
$G=(V(G),E(G))$ consists of 
a set $V(G)$ of \textit{vertices} and 
a multiset $E(G)$ of \textit{directed edges}. 
The structure is characterized as follows:
\begin{itemize}
\item Each edge $e \in E(G)$ is an ordered pair 
$(u,v) \in V(G) \times V(G)$,
representing an edge that originates from vertex $u$
and terminates at $v$.
This is denoted by $o(e)=u$ and $t(e)=v$.
\item 
If $(u,v) \in E(G)$ and $(v,u) \in E(G)$ these edges are 
considered \textit{inverses} of each other, often denoted by 
$e=(u,v)$ and $e^{-1}=(v,u)$.
    \item A directed edge with $u=v$ is called a \textit{directed loop}.
\item The digraph $G$
is \textit{finite} if both $V(G)$ and $E(G)$ are finite.
\item The digraph $G$ has \textit{multiple directed edges} if $E(G)$
contains multiple copies of the same ordered pair.
\item The digraph $G$ is called a \textit{simple digraph} 
if it has neither directed loops nor multiple directed edges.
\end{itemize}
\end{Def}

Next, we introduce the notion of connectivity for digraphs.

\begin{Def}
Let $G=(V(G),E(G))$ be a digraph.
A digraph G is said to be \textit{(weakly) connected} 
if for every pair of distinct vertices $s,t \in V(G)$,
there exists a finite sequence of vertices
$(v_0,v_1,\dots,v_n)$ of $V(G)$ such that 
$v_0=s, \ v_n=t$, and for each $i\in \{0,1,\dots,n-1\}$,
either $(v_i,v_{i+1}) \in E(G)$ or $(v_{i+1},v_i) \in E(G)$ holds.
\end{Def}

Next, we define paths and cycles in digraphs,
which are essential for defining the Ihara zeta function.

\begin{Def}\label{def:simplecycle}
    Let $G$ be a finite, simple and connected digraph. 
    The path on $G$ is defined as follows:
    \begin{itemize}
        \item A \textit{path} $P=(e_1,e_2,\dots,e_n)$
        of length $n$ on $G$ 
        is a finite sequence of edges $e_j \in E(G)$ such that 
        $t(e_k)=o(e_{k+1})$
        for all $k \in \{1,\dots,n-1\}$. The length of $P$ is denoted by $l(P)=n$.
        \item A path $P=(e_1,e_2,\dots,e_n)$ on $G$
        is called a \textit{cycle} if $t(e_n)=o(e_1)$.
        \item A cycle $C=(e_1,e_2,\dots,e_n)$ on $G$ 
        is called \textit{prime} if it is not a power of any 
        shorter cycles; that is 
        $C\neq B^r$, for any cycle $B$ and any integer $r \ge 2$.
        \item A cycle $C=(e_1,e_2,\dots,e_n)$ on $G$
        is called \textit{reduced} if 
        $e_{k+1}^{-1} \neq e_k$ for all $k\in \{1,\dots,n\}$.
        However $e_{n+1}=e_1$.
        \item A cycle $C=(e_1,e_2,\dots,e_n)$ on $G$
        is called {\it simple} if all of its vertices are distinct, i.e., 
        $\#\{o(e_1),o(e_2),\dots,o(e_n)\}=n$
    \end{itemize}
\end{Def}

\begin{Def}
    Let $G$ be a finite, simple and connected digraph.
    Two cycles $C_1=(e_0,e_1,\dots,e_{n-1})$ and 
    $C_2=(f_0,f_1,\dots,f_{n-1})$ of the same length $n$
    are said to be equivalent 
    if there exists some $k \in \NN$ such that 
    $e_{j+k \pmod n}=f_j$ for all $j=0,\dots,n-1$.
    The equivalence class of a cycle $C$ is denoted by $[C]$.
\end{Def}

Next, we provide the definition of Ihara zeta functions for digraphs.

\begin{Def}[\cite{KS00}]
    Let $G$ be a finite, simple and connected digraph.
    The Ihara zeta function $\zeta(G,u)$ of $G$ is defined as follows:
    \begin{equation*}
        \zeta(G,u)=\prod_{[C]}\left(1-u^{l(C)}\right)^{-1},
    \end{equation*}
    where the product runs over all equivalence classes of prime and reduced 
    cycles on $G$.
\end{Def}

We provide the definition of the line digraph of $G$
(see, e.g., \cite{BB21}).

\begin{Def}
    Let $G=(V(G),E(G))$ be a finite, simple and connected digraph.
    The {\it line digraph} of $G$, denoted by $L(G)$, 
    is the digraph defined as follows:
    \begin{itemize}
        \item $V(L(G))=E(G)$,
        \item $E(L(G))=\{(e,f) 
        \mid e,f \in E(G),\  t(e)=o(f) \}$.
    \end{itemize}
\end{Def}

\begin{Def}\label{def:LWBG}
    Let $G$ be a finite, simple and  connected digraphs.
    The {\it line digraph without backtrack} of $G$, denoted by $L_{WB}(G)$,
    is the digraph defined as follows:
    \begin{itemize}
        \item $V(L_{WB}(G))=E(G)$,
        \item $E(L_{WB}(G))=\{(e,f) 
        \mid e,f \in E(G), \ t(e)=o(f),\  e^{-1} \neq f \}$.
    \end{itemize}
\end{Def}

\begin{rem}
The structure we define as a "line digraph without backtracks"
corresponds to the \textit{oriented line graph} in \cite{KS00}.
In this article, we deliberately avoid the latter term 
to prevent confusion with the standard \textit{line digraph}.
Instead, we adopt a more explicit terminology 
to ensure a clear distinction between these two concepts.
\end{rem}

The adjacency matrix of a digraph is fundamental 
to describing the determinant expression of its Ihara zeta function.

\begin{Def}
    Let $G$ be a simple digraph with $n$ vertices.
    The {\it adjacency matrix} of $G$, $A(G) \in M_n(\ZZ)$ 
    is defined as follows:
    $A(G) =(a_{ij})_{1 \le i,j \le n}$ and 
    \begin{equation*}
        a_{ij}=\begin{cases}
            1 & \text{if}\  (v_i,v_j) \in E(G),\\
            0 & \text{otherwise}.
        \end{cases}
    \end{equation*}
\end{Def}

The reciprocal of the Ihara zeta function for a digraph
can be expressed via several determinant formulae.
Here, we present one such expression, 
known as the \textit{Hashimoto-type expression}.

\begin{prop}[\cite{KS00}]\label{prop:hashimoto}
Let $G$ be a finite, simple and connected digraph, 
and let $\zeta(G,u)$ be its Ihara zeta function. 
Then, the following determinant formula holds:
\begin{equation*}
\zeta(G,u)^{-1} = \det(I - A(L_{WB}(G))u),
\end{equation*}
where $L_{WB}(G)$ is the line digraph without backtrack of $G$,
and $A(L_{WB}(G))$ is its adjacency matrix.
\end{prop}
For a detailed proof of this proposition,
we refer the reader to, for example, \cite{FZ99,KS00}.

\section{Trace monoids}

In this section,
we introduce the notion of 
a trace monoid (a.k.a partially commutative monoid, graph monoid).
The notion of trace monoid 
was introduced by Cartier and Foata in \cite{CF69}.

\begin{Def}
    Let $A$ be a (finite or infinite) set.
    The free monoid generated by $A$ is denoted as $A^*$.
    Each element of $A^*$ is called a word.
\end{Def}

We define trace monoid in the following definition.

\begin{Def}{[cf. \cite{DuchampKrob92,K09,König91}]}\label{def:trace_monoid}
Let $A$ be a finite set.
    We introduce a {\it commutation relation} $I \subset A\times A$, which is 
    irreflexive and symmetric; 
    that is, $(a,a) \not\in I$
    for all $a \in A$, and 
    $(a,b) \in I$ implies $(b,a) \in I$.
    This relation induces the congruence $\equiv_I$ on the 
    free monoid $A^*$ such that $ab\equiv_I ba$ for all $(a,b) \in I$.
    The quotient monoid $A^*/\equiv_I$ is called a 
    \textit{trace monoid}
    (or \textit{partially commutative monoid, graph monoid}).
    Each element of the trace monoid is called 
    a \textit{trace}, representing an equivalence class of words in $A^*$.
\end{Def}

Following convention, we shall denote the quotient monoid 
$A^*/ \equiv_I$ simply as $A^*/I$

\begin{rem}
Let $M=A^*/I$ be the trace monoid defined by a finite set $A$
and a commutation relation $I$.
The multiplication in $M$ is induced directly from 
the concatenation in the free monoid $A^*$.
That is, for any traces $[w],[v] \in M$,
their product is given by $[w]\cdot [v]=[wv]$.
This ensures that the following diagram commutes:
\begin{center}
\begin{tikzcd}
A^* \times A^* \arrow[rr] \arrow[d, "\pi \times \pi"'] &  & A^* \arrow[d, "\pi"] \\
A^*/I \times A^*/I \arrow[rr]                          &  & A^*/I               
\end{tikzcd}
\end{center}
where $\pi:A^* \to M$ is the canonical projection.
\end{rem}

We define the monoid algebra associated 
with a trace monoid in order to consider
formal power series and polynomials.

\begin{Def}[\cite{DK93}]
Let $M=A^*/I$ be a trace monoid.
We denote by $ K=\ZZ \llangle A \rrangle $ (resp. $\ZZ \langle A \rangle $) 
the $\ZZ$-algebra of non-commutaive 
formal power series (resp. polynomials) 
over $A$.
Let $J_I$ be the two-sided ideal of $K$ 
generated by the elements 
$\{(ab-ba)  \mid (a,b) \in I\}$.
The partially commutative algebra of formal power series over $M$,
denoted by $\ZZ \llangle M \rrangle$, 
is defined as the quotient $K / J_I$.
The polynomial ring $\ZZ \langle M \rangle $ is similarly defined as 
$\ZZ \langle A \rangle /J_I$.
\end{Def}

We define two graphs associated with a trace monoid as follows.

\begin{Def}[\cite{König91}]
Let $M=A^*/I$ be a trace monoid.
    The {\it commutation graph} 
    $G(M)$ of $M$,
    is the (undirected) simple graph defined as follows:
    \begin{itemize}
        \item The set of vertices $V(G(M))=A$,
        \item The set of edges $E(G(M))=\{\{a,b\} \mid (a,b) \in I\}$.
    \end{itemize}
\end{Def}

\begin{Def}[\cite{König91}]\label{def:noncomgraph}
Let $M=A^*/I$ be a trace monoid.
    The {\it noncommutation graph} 
    $\Tilde{G}(M)$ of $M$,
    is the (undirected) graph defined as follows:
    \begin{itemize}
        \item The set of vertices $V(\Tilde{G}(M))=A$,
        \item The set of edges $E(\Tilde{G}(M))=\{\{a,b\} \mid (a,b) \notin I \}$.
    \end{itemize}
\end{Def}

The noncommutation graph has loops, thus it is not simple.

\begin{ex}\label{ex:abc}
    Let $A=\{a,b,c\}$ and $I=\{(b,c),(c,b)\}$.
    Then the word $abc \in A^*$ corresponds to the trace 
     $[abc] \in M=A^*/I$ which is the set $\{abc,acb\}$.
     The commutation graph and noncommutation graph are as follows:
     \begin{figure}[ht]
    \centering
    \begin{minipage}{0.48\textwidth}
        \begin{tikzpicture}[
    thick, 
    auto, 
    every edge/.style={draw, thick, shorten >=1.5pt, shorten <=1.5pt},
    every node/.style={draw, thick, circle, inner sep=1pt, minimum size=0.6cm, font=\small\bfseries},
    scale=0.9, transform shape]

    \node (1) at (90:2.5cm) {a};
    \node (3) at (-30:2.5cm) {b};
    \node (5) at (-150:2.5cm) {c};

    


    \draw (3) to (5);

\end{tikzpicture}
\caption{The commutation graph $G(M)$}
    \end{minipage}
\begin{minipage}{0.48\textwidth}
    \begin{tikzpicture}[
    thick, 
    auto, 
    every edge/.style={draw, thick, shorten >=1.5pt, shorten <=1.5pt},
    every node/.style={draw, thick, circle, inner sep=1pt, minimum size=0.6cm, font=\small\bfseries},
    scale=0.9, transform shape]

    \node (1) at (90:2.5cm) {a};
    \node (3) at (-30:2.5cm) {b};
    \node (5) at (-150:2.5cm) {c};

    
    \draw (1) to (3);
    \draw (1) to (5);

    \draw [thick, out=60,   in=120,  min distance=1cm] (1) to (1);
    \draw [thick, out=-60,  in=0,    min distance=1cm] (3) to (3);
    \draw [thick, out=-180, in=-120, min distance=1cm] (5) to (5);
    
\end{tikzpicture}
\caption{The noncommutation graph $\tilde{G}(M)$}
\end{minipage}
\end{figure}
\end{ex}

We define the notion of a clique and 
introduce its associated clique polynomial.

First, We define complete graphs and subgraphs, 
which are essential for introducing the concept of a clique.
\begin{Def}
    A finite and simple graph $G=(V(G),E(G))$
    is called a \textit{complete graph} if 
    every pair of distinct vertices is connected by an edge; that is, 
    $E(G)=\{\{u,v\} \mid u,v \in V(G), \ u \neq v\}$.
    A complete graph with $n$ vertices is denoted by $K_n$.
\end{Def}

\begin{Def}
Let $G=(V(G),E(G))$ be a graph.
    A graph $H=(V(H),E(H))$ is called 
    an \textit{induced subgraph} of $G$ if  
    $V(H) \subset V(G)$ and 
    $E(H)= E(G) \cap  \{\{u,v\} \mid u,v \in V(H)\}$.
\end{Def}

We now define a clique.
\begin{Def}
Let $G$ be a finite and simple graph.
    An induced subgraph $H$ of graph $G$ is called a {\it clique} of size $s$
    if $H$ is a complete graph $K_s$.
    We denote by $c_s(G)$ the number of cliques of size $s$ in $G$:
    \begin{equation*}
        c_s(G)\coloneqq
        \# \{H \mid H \text{ is a clique of size} \ s \ \text{in} \ G\}.
    \end{equation*}
    By convention, we set $c_0(G)=1$ (representing the empty set $\emptyset$).
    In particular, we have $c_1(G)=\#V(G)$.
\end{Def}

\begin{rem}
Let $M=A^*/I$ be a trace monoid.
    If the set $A$ is finite, then both graphs 
    $G(M), \Tilde{G}(M)$ are finite graphs.
    Furthermore, since the commutation relation $I$ is irreflexive 
    by Definition \ref{def:trace_monoid},
    then the commutation graph $G(M)$ is simple, as it contains 
    no loops or multiple edges. 
\end{rem}

We introduce clique polynomials.

\begin{Def}[\cite{LR09}]\label{def:cliquepolyonZ}
Let $G$ be a finite and simple graph with $n$ vertices.
    The \textit{clique polynomial} of the graph $G$,
    denoted by $p(G,u) \in \ZZ[u]$, is 
    defined as follows:
    \begin{equation*}
        p(G,u)=\sum_{k=0}^n (-1)^k c_k(G)u^k.
    \end{equation*}
\end{Def}

\begin{Def}[\cite{BG11,LR09}]\label{def:cliquepolyont}
Let $G$ be a finite and simple graph where each vertex $v_i$ 
is assigned a positive integer weight $k_i \in \ZZ_{>0}$.
For any subgraph $H$, its weight is defined by 
$\mathrm{weight}(H)=\sum\limits_{v_j \in V(H)}k_j$.
    The {\it vertex weighted clique polynomial} $p_w(G,u) \in 
    \ZZ[u^{k_1},\dots,u^{k_n}] (\subset \ZZ[u])$
    is defined as follows:
    \begin{equation*}
        p_w(G,u)=\sum_{B} (-1)^{\#V(B)} u^{\mathrm{weight}(B)},
    \end{equation*}
    where the sum runs over all cliques $B$ of $G$ 
    (including the empty clique with
    $\mathrm{weight}(\emptyset)=0$).
\end{Def}

For simplicity, we henceforth denote the clique polynomials $p(G,u)$ and
$p_w(G,u)$ purely as $p(G)$ and $p_w(G)$, respectively,
unless the explicit presence of the parameter $u$ is necessary.

We introduce the clique polynomial defined over a trace monoid as follows:

\begin{Def}[\cite{König91}]\label{def:cliquepolyonM}
Let $M=A^*/I$ be a trace monoid with a finite generating set 
$A=\{a_1,\dots,a_n\}$.
    Let $G(M)$ be the commutation graph of $M$.
    Let $\mathcal{C}(G(M))$ be the set of all 
    cliques of $G(M)$.
    Let $\phi:\mathcal{C}(G(M)) \to M$ be the map that 
    sends each clique $H$ with vertices  
     $V(H)=\{a_{i_1},\dots,a_{i_k}\}$ to 
     the trace $[a_{i_1} \cdots a_{i_k}] \in M$,
     which is well-defined because all vertices in $V(H)$
     pairwise commute in $M$.
    The {\it clique polynomial} $p_M(G(M))$ of 
    the graph $G(M)$ on $\ZZ \langle M \rangle$
    is defined as follows:
    \begin{equation*}
        p_M(G(M))=\sum_{ H \in \mathcal{C}(G(M))} (-1)^{\#V(H)} \phi(H),
    \end{equation*}
    where $\phi(\emptyset)=1$ (the identity element of $M$).
\end{Def}

\begin{rem}
    We supposed that in 
    Definition \ref{def:cliquepolyonZ}, 
    Definition \ref{def:cliquepolyont}, and 
    Definition \ref{def:cliquepolyonM}
    the graph is assumed to be finite.
    This assumption ensures that clique polynomials 
    $p(G)$, $p_w(G)$ and $p_M(G(M))$ 
    are well-defined as sums of finitely many terms.
\end{rem}

The following lemma clarifies the relationship 
between the various definitions of clique polynomials 
by showing that they are connected 
via specific ring homomorphisms.
In particular, it illustrates
how the non-commutative polynomial over a trace monoid
reduces to the standard clique polynomial 
by specializing the weights of the generators.

\begin{lem}\label{lem:psi}
Let $M=A^*/I$ be a trace monoid with a finite generating set 
$A=\{a_1,\dots,a_n\}$, where each $a_i$ 
is assigned a positive integer weight $k_i$.
Let $G(M)$ be the commutation graph of $M$.
Let $p_M(G(M))$, $p_w(G(M))$ and $p(G(M))$ be the clique polynomials on 
$\ZZ \langle M \rangle ,\ZZ[u^{k_1},\dots,u^{k_n}],\ZZ[u]$ respectively.
Let 
$\psi_1: \ZZ \langle M \rangle \to \ZZ[u^{k_1},\dots, u^{k_n}]$
and $\psi_2: \ZZ[u^{k_1},\dots, u^{k_n}] \to \ZZ[u]$
be the ring homomorphisms uniquely determined by 
$\psi_1(a_i) = u^{k_i}$ and $\psi_2(u^{k_i}) = u$ for all $1 \le i \le n$.
Then, the following identities hold:
    \begin{equation*}
        \psi_1(p_M(G(M)))=p_w(G(M)), 
        \quad \text{and} \quad  \psi_2(p_w(G(M)))=p(G(M)).
    \end{equation*}
\end{lem}
\begin{proof}
The assertions follow by tracking
the images of the defining sums under each homomorphism.
Applying $\psi_1$ term-by-term to  
$p_M(G(M))$ (Definition \ref{def:cliquepolyonM}) 
directly yields $p_w(G(M))$ (Definition \ref{def:cliquepolyont}).
Similarly, applying $\psi_2$ term-by-term to 
$p_w(G(M))$ yields $p(G(M))$ (Definition \ref{def:cliquepolyonZ}).
\end{proof}

\section{Trace monoids on Ihara zeta functions}

In this section, 
we investigate the connection between the Ihara zeta function 
and trace monoids, 
a relationship implicitly discussed in \cite{GR17}.
We extend this framework by incorporating the notion of clique polynomials. 
Following the approach of Giscard and Rochet \cite{GR17},
we construct a specific trace monoid defined over the edge set of a graph.

\begin{Def}
    Let $G=(V(G),E(G))$ be a finite, simple and connected digraph.
    Let $L_{WB}(G)$ be the line digraph without backtrack of $G$.
    We define a commutation relation 
    $I \subset E(G)\times E(G)$ as follows:
    \begin{equation*}
    (e, f) \in I \iff o(e) \neq o(f) \  \text{and}\  \ e^{-1} \neq f,
    \end{equation*}
    where $o(e)$ denotes the origin of an edge $e \in E(G)$.
    In other words,
    two edges commute if and only if they do not share
    the same starting vertex and are not inverses of each other.
\end{Def}

To connect the combinatorial properties of graphs
with the algebraic structure of the trace monoid $M=E(G)^*/I$,
we first formalize the notions of paths and cycles.
While these are standard in graph theory,
their representation in $M$ requires careful definition
because the commutation relation $I$ 
allows edges with distinct origins to commute.

\begin{Def}\label{def:tracepath}
Let $G=(V(G),E(G))$ be a finite, simple and connected digraph.
Let $M=E(G)^*/I$ be the trace monoid with the commutation relation 
$(e, f) \in I \iff o(e) \neq o(f)$, and  $e^{-1} \neq f$.
\begin{enumerate}
\item 
A trace $p \in M$ is called a \textit{reduced path} 
if it can be represented by a word $e_1 \cdots e_k \in E(G)^*$ such that $t(e_j)=o(e_{j+1})$ and $e_j^{-1} \neq e_{j+1}$ for all $1\le j < k$.
\item 
A reduced path $p \in M$
has a word representation $e_1 \cdots e_k \in E(G)^*$ 
satisfying $t(e_i)=o(e_{i+1})$ and $e_i^{-1} \neq e_{i+1}$ 
for all $1 \le i < k$. 
It is called a \textit{cycle} if $t(e_k)=o(e_1)$ and $e_k^{-1} \neq e_1$.

\item 
A cycle $p=[e_1 \cdots e_k] \in M$ is called \textit{simple} 
if all edges $e_1, \dots, e_k$ are distinct.

\end{enumerate}
\end{Def}

\begin{rem}\label{rem:subdivision}
\begin{enumerate}
\item This definition of a simple cycle becomes particularly natural when interpreted on the line digraph without backtrack $L_{WB}(G)$. 

\item By Definition \ref{def:tracepath}, 
the number of all simple cycles is 
finite if the graph $G$ finite.

\item The number of all simple cycles is unchanged 
under the subdivision of directed edges without an inverse.
This is because such an operation simply replaces a single directed edge
with a path of two new directed edges,
preserving the sequential structure of simple cycles.
\end{enumerate}
\end{rem}

Giscard and Rochet \cite{GR17} established 
the following decomposition property within this framework:
\begin{lem}[\cite{GR17}]
    Let $G=(V(G),E(G))$ be a finite, simple and connected digraph.
    Let $M=E(G)^*/I$ be the trace monoid with the commutation relation 
    $(e, f) \in I \iff o(e) \neq o(f)$ and $e^{-1}\neq f$.
    Then any cycle in $M$ is the product of simple cycles.
\end{lem}

This cycle decomposition property 
motivates us to consider an algebraic structure 
generated by simple cycles.
Since the structural interactions and dependencies among cycles 
are captured by how they share vertices in $L_{WB}(G)$,
we are led to define a new trace monoid
whose commutation relation is governed
by the vertex sets of these simple cycles.

\begin{Def}\label{def:hikemonoid}
    Let $G=(V(G),E(G))$ be a finite, simple and connected digraph.
    Let $L_{WB}(G)$ be the line digraph without backtrack of $G$.
    Let $M=E(G)^*/I$ be the trace monoid with the commutation relation 
    $(e, f) \in I \iff o(e) \neq o(f)$ and $e^{-1} \neq f$.
The \textit{hike monoid} $\mathcal{H}$ is defined as the trace monoid $\Sigma_M^*/I_{\mathcal{H}}$, which is  generated by the set of all simple cycles $\Sigma_M=\{h_1,\dots,h_k\}$ in $M$. Here, the commutation relation $I_{\mathcal{H}}$ is given by 
\begin{equation*}
    (h_i, h_j) \in I_{\mathcal{H}} \iff V(h_i) \cap V(h_j) = \emptyset,
\end{equation*}
where $V(h_i)$ denotes the set of vertices in $L_{WB}(G)$ visited by $h_i$. 

An element of $\mathcal{H}$ is called a \textit{hike}, and each generator $h_i \in \Sigma_M$ is called a \textit{prime hike}. Two prime hikes $h_i, h_j$ are said to be \textit{disjoint} if $(h_i, h_j) \in I_{\mathcal{H}}$.

\end{Def}

By construction, the hike monoid $\mathcal{H}$ 
is a trace monoid over the alphabet of simple cycles $\Sigma_M$.
Consequently, it possesses an associated commutation graph 
denoted by $G(\mathcal{H})$, 
allowing us to study its algebraic properties 
through the language of graph theory.

\begin{rem}\label{rem:submonoid}
Each simple cycle $h_i \in \Sigma_M$ 
is naturally an element of $M$. 
If two simple cycles $h_i, h_j$ are vertex-disjoint in $L_{WB}(G)$,
every edge in $h_i$ has a different origin from every edge in $h_j$.
Hence, the defining commutation relation in $\mathcal{H}$
agrees with the relation inherited from $M$:
\begin{equation*}
    (h_i, h_j) \in I_{\mathcal{H}} \quad \Longrightarrow \quad h_i h_j = h_j h_i \quad \text{in} \   M.
\end{equation*}
Therefore, the map sending each generator $h_i$ of $\mathcal{H}$ to the corresponding element of $M$ extends to an injective monoid homomorphism $\mathcal{H} \hookrightarrow M$. In this sense, $\mathcal{H}$ can be naturally identified with the submonoid of $M$ generated by all simple cycles. This factorization viewpoint makes the structure of hikes transparent.

\end{rem}

The following proposition establishes the connection between the Hashimoto-type determinant formula and hike monoid.

\begin{prop}[cf.\ \cite{P66}]\label{prop:ponst}
Let $G=(V(G),E(G))$ be a finite, simple and connected digraph,
and let $L_{WB}(G)$ be its line digraph without backtracks. 
Under the following notation, the determinant satisfies the identity below:
\begin{itemize}
    \item $\mathcal{H}$: the hike monoid associated with $L_{WB}(G)$,
    \item $\mathcal{C}(G(\mathcal{H}))$: the set of all cliques in the commutation graph $G(\mathcal{H})$ (i.e., each $H \in \mathcal{C}(G(\mathcal{H}))$ is a set of pairwise disjoint simple cycles).
\end{itemize}
\begin{equation*}
\det (I - A(L_{WB}(G))u) = \sum_{H \in \mathcal{C}(G(\mathcal{H}))} (-1)^{\#H} u^{l(H)},
\end{equation*}
where $\#H$ denotes the number of simple cycles in the clique $H$, and $l(H)=\sum_{h \in H}l(h)$ is the total length of the cycles in $H$.
\end{prop}

\begin{rem}\label{rem:commutaiongraph}
In the setting of Proposition \ref{prop:ponst}, 
the commutation graph $G(\mathcal{H})$ can be viewed
as a vertex-weighted graph. 
Specifically, each vertex $h_i$ in $G(\mathcal{H})$ 
corresponds to a simple cycle of $L_{WB}(G)$,
and we assign the length $l(h_i)$ as its weight.
In this sense, the determinant formula in Proposition \ref{prop:ponst}
is precisely the clique polynomial of $G(\mathcal{H})$
weighted by simple cycle lengths.
\end{rem}

We summarize this connection in the following lemma.

\begin{lem}\label{lem:hashiwei}
In the setting of Proposition \ref{prop:ponst},
the determinant satisfies the following identity:
\begin{equation*}
\det (I - A(L_{WB}(G))u) = p_{w}(G(\mathcal{H}),u),
\end{equation*}
where $p_w(G(\mathcal{H}),u)$ is the vertex-weighted clique polynomial 
of the commutation graph $G(\mathcal{H})$ 
with respect to the weights given by the cycle lengths $k_i=l(h_i)$.
\end{lem}
\begin{proof}
By Remark \ref{rem:commutaiongraph},
the commutation graph $G(\mathcal{H})$ 
is a vertex-weighted graph where each vertex $h_i$ is
assigned the weight $l(h_i)$.
Comparing the expansion of the clique polynomial in 
Definition \ref{def:cliquepolyont} with the Hashimoto-type formula 
in Proposition \ref{prop:ponst},
we see that each term in the sum 
corresponds to a clique of disjoint simple cycles,
weighted by their total length. This yields the desired equality.
\end{proof}

\begin{rem}
    As in \cite[Introduction]{LR09}, inverses of 
     vertex weighted clique polynomials are the Hilbert series 
    of the graded free partially commutative monoids.
    Consequently, Lemma \ref{lem:hashiwei} provides a novel algebraic interpretation of the Ihara zeta function as a Hilbert series.
\end{rem}

This algebraic characterization 
implies a remarkable geometric property:
if two digraphs have isomorphic vertex-weighted commutation graphs
$G(\mathcal{H})$, they must share the exact same Ihara zeta function.
To illustrate this phenomenon, a concrete example is provided below.

\begin{ex}
Consider the two non-isomorphic digraphs,
$G_1$ and $G_2$, shown in Figure \ref{fig:digraph_g1} and 
Figure \ref{fig:digraph_g2}:

\begin{figure}[ht]

\begin{center}
  \begin{minipage}[b]{0.45\textwidth}
    \centering
    \begin{tikzcd}[row sep=large, column sep=large]
    \bullet \arrow[dd] & & & & \bullet \arrow[dd] \\
    & \bullet \arrow[lu] \arrow[rr, bend left] & & \bullet \arrow[ll, bend left] \arrow[ru] & \\
    \bullet \arrow[ru] & & & & \bullet \arrow[lu]
    \end{tikzcd}
    \vglue 0.2cm
    \caption{Digraph $G_1$}
    \label{fig:digraph_g1}
  \end{minipage}
  \hfill
  \begin{minipage}[b]{0.45\textwidth}
    \centering
    \begin{tikzcd}[row sep=large, column sep=large]
    & \bullet \arrow[dd] \arrow[r] & \bullet \arrow[r] & \bullet \arrow[rd] & \\
    \bullet \arrow[ru] & & & & \bullet \arrow[ld] \\
    & \bullet \arrow[lu] & \bullet \arrow[l] & \bullet \arrow[l] \arrow[uu] & 
    \end{tikzcd}
    \vglue 0.2cm
    \caption{Digraph $G_2$}
    \label{fig:digraph_g2}
  \end{minipage}
\end{center}

\end{figure}

In this case, their vertex-weighted commutation graphs are identical. 
Specifically, both digraphs possess exactly 
three simple cycles $\{c_1, c_2, c_3\}$ in their respective
line digraphs without backtrack $L_{WB}(G)$,
with lengths given by $l(c_1)=3$, $l(c_2)=3$, and $l(c_3)=8$.

Furthermore, the commutation graph $G(\mathcal{H})$ of both digraphs is isomorphic to the graph $G(M)$ presented in Example \ref{ex:abc} under the vertex identification $a \mapsto c_3$, $b \mapsto c_1$, and $c \mapsto c_2$. Under this structure, the set of all cliques of $G(\mathcal{H})$ is explicitly given by:
\begin{equation*}
    \mathcal{C}(G(\mathcal{H})) = \{ \emptyset, \{c_1\}, \{c_2\}, \{c_3\}, \{c_1, c_2\} \}.
\end{equation*}
Consequently, the reciprocal of the Ihara zeta function for both graphs is computed straightforwardly as follows:
\begin{align*}
    \zeta(G, u)^{-1}
    &= \det(I - A(L_{WB}(G))u) && (\text{by Proposition \ref{prop:hashimoto}}) \\
    &= p_w(G(\mathcal{H}),u)   && (\text{by Lemma \ref{lem:hashiwei}}) \\
    &= 1 - 2u^3 + u^6 - u^8    && (\text{by Definition \ref{def:cliquepolyont}}).
\end{align*}

\end{ex}

To generalize our results to the non-commutative setting 
of the monoid rings $\ZZ \langle M \rangle$ and 
$\ZZ \langle \HH \rangle$,
we introduce a non-commutative generalization of the determinant

\begin{Def}[{cf. \cite{AS18}}]\label{def:Cdet}
Let $R$ be a (possibly non-commutative) ring.
    For any matrix $X=(x_{i,j})_{1 \le i,j \le n} \in M_n(R)$, the 
    {\it Cayley determinant} is defined as follows:
    \begin{equation*}
        \mathrm{Cdet}(X)=\sum_{\sigma \in S_n}\mathrm{sgn}(\sigma)x_{1,\sigma(1)}\cdot x_{2,\sigma(2)} \cdots x_{n,\sigma(n)},
    \end{equation*}
    where the product of the entries is taken 
    in the increasing order of their row indices.
\end{Def}

\begin{rem}
If $R$ is commutative in Definition \ref{def:Cdet}, 
the Cayley determinant reduces to the usual determinant.
\end{rem}

\begin{prop}[cf.\cite{GR17}]\label{prop:hashimotohike}
Let $M=E(G)^*/I$ be the trace monoid defined 
by the commutation relation $(e, f) \in I \iff o(e) \neq o(f)$ and
$e^{-1} \neq f$. 
In the setting of Proposition \ref{prop:ponst},
the Cayley determinant satisfies the following identity
in the monoid ring $\ZZ \langle M \rangle$:
\begin{equation*}
    \mathrm{Cdet}(I - W) = p_{\HH}(G(\HH)),
\end{equation*}
where $W = (w_{ef})_{e,f \in E(G)}$ 
is the edge-weighted adjacency matrix defined by 
$w_{ef} = [e]$ if $t(e) = o(f)$ and $e^{-1} \neq f$, 
and $w_{ef} = 0$ otherwise, and 
$p_{\mathcal{H}}(G(\mathcal{H}))$ is the 
clique polynomial of the commutation graph $G(\mathcal{H})$.

\end{prop}

\begin{rem}\label{rem:12(3)}
Regarding Proposition \ref{prop:hashimotohike}, 
two algebraic subtleties should be noted:
\begin{enumerate}
    \item As discussed in Remark \ref{rem:submonoid},
    the injective monoid homomorphism $\mathcal{H} \hookrightarrow M$
    naturally induces a ring injection
    $\ZZ \langle \mathcal{H} \rangle \hookrightarrow \ZZ \langle M \rangle$. 
    Under this identification, 
    while the matrix $W$ is defined over the larger ring 
    $\ZZ \langle M \rangle$, the image of the Cayley determinant
    $\mathrm{Cdet}(I - W)$ belongs entirely to the subring
    $\ZZ \langle \mathcal{H} \rangle$.
    In other words, the algebraic identity holds 
    strictly within the smaller monoid ring $\ZZ \langle \mathcal{H} \rangle$.
    
    \item While the original formulation in \cite{GR17}
    employs the standard determinant notation $\det$,
    we intentionally use the Cayley determinant $\mathrm{Cdet}$
    to properly accommodate the non-commutative structure
    of the monoid ring $\ZZ \langle M \rangle$.
    Although the multiplication of edges is generally non-commutative in $M$, 
    any two vertex-disjoint simple cycles appearing in the expansion
    of $p_{\mathcal{H}}(G(\mathcal{H}))$ commute with each other
    by the definition of $I_{\mathcal{H}}$.
    Thus, the use of $\mathrm{Cdet}$ ensures a rigorous algebraic 
    treatment for non-commutative alphabets 
    while remaining perfectly consistent 
    with the commutative reduction in \cite{GR17}.

    \item \label{rem:(3)}
    The non-commutative identity 
    in Proposition \ref{prop:hashimotohike} 
    recovers the commutative relation in Lemma \ref{lem:hashiwei}
    as a special case. Specifically, by applying 
    the ring homomorphism $\psi_1$ from Lemma \ref{lem:psi} 
    to the hike monoid $\HH = \Sigma_M^*/I_{\mathcal{H}}$
    with weights given by cycle lengths $l(h_i)$, 
    the identity $\psi_1(p_{\HH}(G(\HH))) = p_w(G(\HH))$
    follows immediately. 
    This reduction demonstrates that our non-commutative framework
    naturally generalizes the classical commutative setting
\end{enumerate}
\end{rem}

\section{Main theorem}
In this section, 
we present and prove our main results through two principal theorems.
We begin by constructing the closure of the noncommutation graph
associated with the hike monoid. 
From the combinatorial structure of this digraph,
we first derive a general determinant formula over the hike monoid ring.
We then apply a ring homomorphism $\psi_1$ from Lemma \ref{lem:psi} 
to yield a determinant expression 
for the reciprocal of the Ihara zeta function.

We use the notation below to define the closure:
\begin{itemize}
    \item Let $G$ be a finite, simple and connected digraph,
    \item Let $L_{WB}(G)$ be a line digraph without backtrack of $G$,
    \item Let $M=E(G)^*/I$ be the trace monoid,
    \item Let $\mathcal{H}=\Sigma_M^*/I_{\mathcal{H}}$ be the hike monoid,
    \item Let $\Tilde{G}(\mathcal{H})$ be the noncommutation graph of $\mathcal{H}$.
\end{itemize}

We now introduce a directed augmentation of 
$\Tilde{G}(\mathcal H)$.

\begin{Def}\label{def:closure}
A digraph
\[
\Tilde{G}_+(\mathcal H)
=
(V(\Tilde{G}(\mathcal H)),E)
\]
is called a \textit{closure} of $\Tilde{G}(\mathcal H)$
if its edge set $E$ satisfies the following conditions.

\begin{enumerate}
\item
For every vertex $h_i$, the directed loop $(h_i,h_i)$ belongs to $E$.

\item
If $\{h_i,h_j\}\in E(\Tilde{G}(\mathcal H))$ with $i\neq j$,
then both directed edges
\[
(h_i,h_j), \ (h_j,h_i)
\]
belong to $E$.

\item
If $\{h_i,h_j\}\notin E(\Tilde{G}(\mathcal H))$,
then exactly one of
\[
(h_i,h_j), \quad (h_j,h_i)
\]
belongs to $E$.
\end{enumerate}

Moreover, these orientations must satisfy the following consistency conditions
for every triple $\{h_s,h_t,h_u\}$.

\begin{description}
\item[(1)]
If $h_s$ and $h_t$ commute, but $h_u$
commutes with neither of them,
then the orientations are synchronized:
\[
(h_s,h_u),(h_t,h_u)\in E
\]
or
\[
(h_u,h_s),(h_u,h_t)\in E.
\]

\item[(2)]
If no pair among $\{h_s,h_t,h_u\}$ commutes,
then the induced directed triangle is acyclic; namely,
it is not of the form
\[
h_s \to h_t \to h_u \to h_s.
\]

\end{description}
\end{Def}

\begin{rem}
In the case where $\#\Sigma_M=2$,
if the two simple cycles do not commute each other, 
we may assign only one directed edge 
in either direction. 
\end{rem}

\begin{rem}
It should be noted that a closure $\Tilde{G}_+(\HH)$
does not absolutely exist for every commutation graph.
Furthermore, even when a closure exists,
it is generally not unique. 
The precise characterization of the conditions 
under which a commutation graph admits a closure 
remains an open question, and similarly, 
the problem of determining the number of possible closures 
for a given digraph has not yet been resolved.

We also remark that there are infinitely many closurable graphs.
Indeed, as discussed in Remark \ref{rem:subdivision} (3),
the property of being closurable is preserved under graph operations
such as the subdivision of directed edges without inverses,
since this operation alters neither the number of simple cycles nor the commutation relations among them.
\end{rem}

    We present an example of a digraph that its noncommutation graph 
     does not admit a closure.

\begin{ex}
Consider the directed graph $G$ 
shown in Figure \ref{fig:nonclosurable}.
Its line digraph without backtrack, $L_{WB}(G)$, 
is depicted in Figure \ref{fig:noncloLWBG}.

    \begin{figure}[ht]
    \centering
    \begin{minipage}{0.45\textwidth}
    \begin{tikzpicture}[
    >=Stealth,
    vertex/.style={draw=black, circle, fill=black, inner sep=0pt, minimum size=5pt, thick},
    main edge/.style={->, shorten >=6pt, shorten <=2pt, thick, color=black},
    edge label/.style={draw=black, circle, fill=white, inner sep=2pt, font=\large, thick},
    scale=0.8
]

\node[vertex] (v1) at (0, 4) {};
\node[vertex] (v2) at (4, 4) {};
\node[vertex] (v3) at (0, 0) {};
\node[vertex] (v4) at (4, 0) {};


\draw[main edge, bend left=40] (v2) to node[edge label] {1} (v4);

\draw[main edge, bend left=15] (v4) to node[edge label] {5} (v2);

\draw[main edge, bend left=35] (v3) to node[edge label] {3} (v1);

\draw[main edge, bend left=15] (v1) to node[edge label] {7} (v3);

\draw[main edge, bend left=35] (v1) to node[edge label] {4} (v2);

\draw[main edge, bend left=15] (v2) to node[edge label] {6} (v1);

\draw[main edge, bend left=15] (v3) to node[edge label] {8} (v4);

\draw[main edge, bend left=35] (v4) to node[edge label] {2} (v3);

\draw[main edge, bend right=15] (v4) to node[edge label, pos=0.7] {9} (v1);

\draw[main edge, bend right=15] (v1) to node[edge label, pos=0.7] {10} (v4);

\end{tikzpicture}
    \caption{Digraph $G$}
    \label{fig:nonclosurable}
    \end{minipage}
    \hfill
\centering
    \begin{minipage}{0.45\textwidth}
    \begin{tikzpicture}[
    >=Stealth,
    node distance=3.0cm, 
    vertex/.style={draw=black, circle, fill=none, inner sep=2pt, minimum size=14pt, font=\small, thick},
    main edge/.style={->, shorten >=6pt, shorten <=2pt, thick, color=black},
    scale=0.8
]

\node[vertex] (v1) at (-3,  3) {1};
\node[vertex] (v8) at ( 3,  3) {8};
\node[vertex] (v2) at (-3, -3) {2};
\node[vertex] (v5) at ( 3, -3) {5};

\node[vertex] (v9) at ( 0,  2) {9};
\node[vertex] (v4) at (-1.5,  1) {4};
\node[vertex] (v7) at ( 1.5,  1) {7};

\node[vertex] (v10) at ( 0, -2) {10};
\node[vertex] (v3)  at (-1.5, -1) {3};
\node[vertex] (v6)  at ( 1.5, -1) {6};


\draw[main edge] (v1) -- (v2);    
\draw[main edge] (v1) -- (v9);    
\draw[main edge] (v8) -- (v9);    
\draw[main edge] (v8) -- (v5);    
\draw[main edge] (v2) -- (v3);    
\draw[main edge] (v5) -- (v6);    

\draw[main edge] (v9) -- (v4);    
\draw[main edge] (v9) -- (v7);    
\draw[main edge] (v4) -- (v1);    
\draw[main edge] (v7) -- (v8);    

\draw[main edge] (v3) -- (v4);    
\draw[main edge] (v3) -- (v10);   
\draw[main edge] (v6) -- (v7);    
\draw[main edge] (v6) -- (v10);   

\draw[main edge] (v10) -- (v2);   
\draw[main edge] (v10) -- (v5);   

\end{tikzpicture}

    \caption{$L_{WB}(G)$ of $G$}
    \label{fig:noncloLWBG}
    \end{minipage}
\end{figure}

Simple cycles of $L_{WB}(G)$ are listed below,
where each cycle is represented 
by the sequence of vertices in $L_{WB}(G)$ through which it passes:
\begin{enumerate}[$h_1$:]
    \item 1,9,4
    \item 2,3,10
    \item 5,6,10
    \item 7,8,9
    \item 1,2,3,4
    \item 5,6,7,8
    \item 1,2,3,10,5,6,7,8,9,4
    \item 1,9,7,8,5,6,10,2,3,4
\end{enumerate}

Under this setting, 
the corresponding noncommutation graph 
(where each vertex $i$ represents the cycle $h_i$ for visual simplicity)
is constructed as follows:

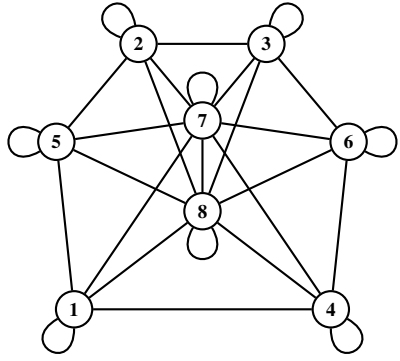
\begin{tikzpicture}[
    thick,
    auto, 
    every edge/.style={draw, thick, shorten >=1.5pt, shorten <=1.5pt},
    every node/.style={draw, thick, circle, inner sep=1pt, minimum size=0.6cm, font=\small\bfseries},
    scale=0.8, transform shape
]

\node (1) at (-135:3.0cm) {1};
\node (4) at (-45:3.0cm)  {4};
\node (6) at (15:2.5cm)   {6};
\node (3) at (65:2.5cm)   {3};
\node (2) at (115:2.5cm)  {2};
\node (5) at (165:2.5cm)  {5};

\node (7) at (0,  1.0) {7};
\node (8) at (0, -0.5) {8};

\draw (1) to (4);
\draw (1) to (5);
\draw (2) to (3);
\node[draw=none] at (115:2.5cm) {}; 
\draw (2) to (5);
\draw (3) to (6);
\draw (4) to (6);

\draw (7) to (1);
\draw (7) to (2);
\draw (7) to (3);
\draw (7) to (4);
\draw (7) to (5);
\draw (7) to (6);
\draw (7) to (8);

\draw (8) to (1);
\draw (8) to (2);
\draw (8) to (3);
\draw (8) to (4);
\draw (8) to (5);
\draw (8) to (6);

\draw [thick, out=-150, in=-90,  min distance=0.8cm] (1) to (1);
\draw [thick, out=100,  in=160,  min distance=0.8cm] (2) to (2);
\draw [thick, out=20,   in=80,   min distance=0.8cm] (3) to (3);
\draw [thick, out=-90,  in=-30,  min distance=0.8cm] (4) to (4);
\draw [thick, out=150,  in=-150, min distance=0.8cm] (5) to (5);
\draw [thick, out=-30,  in=30,   min distance=0.8cm] (6) to (6);
\draw [thick, out=60,  in=120,   min distance=0.8cm] (7) to (7);
\draw [thick, out=-60,  in=-120,   min distance=0.8cm] (8) to (8);

\end{tikzpicture}

One can verify that the induced subgraph on the vertices 
$\{1, 2, 3, 4, 5, 6\}$ is not acyclic on some triangles;
hence, this noncommutation graph cannot have a closure.
\end{ex}

We restrict our discussion to the case
where such an orientation can be consistently assigned.

For simplicity, we establish the following framework,
which will be used throughout the remainder of this section.
\begin{set}\label{set:note}
Let $G=(V(G),E(G))$ be a finite, simple,
and connected digraph, 
and let $L_{WB}(G)$ be its line digraph without backtracks.
We fix the following notation and objects:
\begin{itemize}
    \item $M=E(G)^*/I$: the trace monoid of $G$,
    \item $\Sigma_M$: the set of all simple cycles in $L_{WB}(G)$, with $\# \Sigma_M=n$,
    \item $\mathcal{H}=\Sigma_M^*/I_{\mathcal{H}}$: the hike monoid,
    \item $G(\mathcal{H})$: the commutation graph of $\mathcal{H}$,
    \item $\Tilde{G}(\mathcal{H})$: the noncommutation graph of $\mathcal{H}$,
    \item $\Tilde{G}_+(\HH)$: a closure of $\Tilde{G}(\HH)$,
    \item $\mathcal{C}(G(\mathcal{H}))$: the set of all cliques in $G(\mathcal{H})$,
    \item $l(h)$: the total length of a hike $h = h_1 \dots h_k$ on $L_{WB}(G)$, defined by $l(h)=\sum_{i=1}^k l(h_i)$,
    \item $D(\Tilde{G}_+(\HH))$: the diagonal matrix $\mathrm{diag}[h_1,\dots,h_n] \in M_n(\ZZ\langle \mathcal{H} \rangle)$,
    \item $J$: a subset of indices $\{j_1,\dots,j_k\} \subseteq \{1,\dots,n\}$ sorted such that $j_1 < j_2 < \dots < j_k$,
    \item $\Tilde{G}_+(J)$: the induced subgraph of $\Tilde{G}_+(\HH)$ on the vertices $\{h_j \mid j \in J\}$,
    \item $A(J)$: the adjacency matrix of $\Tilde{G}_+(J)$, whose rows and columns are indexed by $J$,
    \item $D(J)$: the diagonal matrix $\mathrm{diag}[h_{j_1},\dots,h_{j_k}] \in M_k(\ZZ \langle \mathcal{H} \rangle)$.
\end{itemize}
\end{set}

\begin{rem}\label{rem:submat}
In the framework of Setup \ref{set:note}, for any subset of indices 
$J=\{j_1,\dots,j_k\} \subset \{1,\dots,n\}$,
the adjacency matrix $A(J)$ of the induced subgraph $\Tilde{G}_+(J)$
is precisely the submatrix of $A(\Tilde{G}_+(\HH))$ 
formed by selecting the rows and columns corresponding to $J$.
This ensures that the structural properties of $\Tilde{G}_+(\HH)$ are preserved when restricting our focus to any subset of hikes.
\end{rem}

The following lemma reduces the Cayley determinant over the monoid ring 
to a sum of standard determinants, weighted by the corresponding hikes.

\begin{lem}\label{lem:smalldet}
In the framework of Setup \ref{set:note}, 
    the following identity holds:
    \begin{equation}\label{eq:smalldet_form}
        \mathrm{Cdet}(I_n - D(\Tilde{G}_+(\HH))A(\Tilde{G}_+(\HH)) 
        = \sum_{k=0}^n (-1)^k \sum_{1 \le j_1 < \dots < j_k \le n} 
        (h_{j_1} h_{j_2} \cdots h_{j_k})
        \cdot  \det (A(j_1,\dots,j_k) ),
    \end{equation}
    where the term for $k=0$ is defined as $1$.
\end{lem}

\begin{proof}
Let $Q = D(\Tilde{G}_+(\HH))A(\Tilde{G}_+(\HH)$ be an $n \times n$ matrix 
over $\ZZ \langle \HH \rangle$.
We expand the Cayley determinant $\mathrm{Cdet}(I_n - Q)$
by utilizing the multilinearity with respect to its columns.
For each $j \in \{1,\dots, n\}$,
let $q_j$ denote the $j$-th column of $Q$ and 
$e_j$ the $j$-th standard basis vector.
Expanding the determinant column by column, we have:
\begin{align*}
\mathrm{Cdet}(I_n - Q) 
&= \mathrm{Cdet}(e_1 - q_1, e_2 - q_2, \dots, e_n - q_n) \\
&= \sum_{J \subseteq \{1, \dots, n\}}
\mathrm{Cdet}(\alpha_1, \alpha_2, \dots, \alpha_n),
\end{align*}
where the column $\alpha_j$ is chosen as $-q_j$ if $j\in J$,
and $e_j$ otherwise.

For a fixed subset $J = \{j_1, \dots, j_k\}$ with 
$j_1 < j_2 < \cdots < j_k$, the term 
$\mathrm{Cdet}(\alpha_1, \alpha_2, \dots, \alpha_n)$ 
represents the contribution of the submatrix $Q(J)=D(J)A(J)$.
Since $e_j$ acts as the identity in the determinant expansion 
for $j \not\in J$, the expression simplifies to the
Cayley determinant of the $k \times k$ submatrix:
\begin{equation*}
\mathrm{Cdet}(\alpha_1, \dots, \alpha_n) = (-1)^k \mathrm{Cdet}(Q(J)).
\end{equation*}
Recall that $Q(J)=D(J)A(J)$ where $D(J)=
\mathrm{diag}[h_{j_1},\dots,h_{j_k}]$. Because $D(J)A(J)$ 
is a diagonal matrix, 
its entries factor out to the left in the Cayley expansion
in the fixed order of their row indices:
\begin{equation*}
\mathrm{Cdet}(D(J) A(J)) = (h_{j_1} h_{j_2} \cdots h_{j_k}) \mathrm{Cdet}(A(J)).
\end{equation*}
Finally, since $A(J)$ is the adjacency matrix of a
directed graph that satisfies the closure conditions 
(as in Definition \ref{def:closure}),
the non-commutative Cayley determinant $\mathrm{Cdet}(A(J))$
coincides with the standard commutative determinant $\det (A(J))$ 
in its expansion. 
Combining these results for all $J$ yields the desired formula.
\end{proof}

By combining Remark \ref{rem:submat} and Lemma \ref{lem:smalldet},
the computation of the Cayley determinant is reduced to evaluating 
$\det (A(J))$ for all induced subgraphs $\Tilde{G}_+(J)$. 
To this end, we classify the subgraphs based on the
number of vertices $k=\#J$ and their connectivity,
focusing on the following four cases:
\begin{description}
\item[Case (1)] $k=1$: The trivial case of a single vertex.
\item[Case (2)] $k=2$: The case involving a single pair of vertices.
\item[Case (3)] $k\ge 3$ with an isolated edge: the noncommutation graph
$\Tilde{G}(J)$ contains a connected component 
consisting of exactly two vertices $\{j_1,j_2\}$
that are not incident to any other vertices in $J$.
\item[Case (4)] $k\ge 3$ for all other configurations.
\end{description}
In each case, we shall show that $\det (A(J))$ vanishes unless
$J$ forms a clique in $G(\HH)$.

We first evaluate the determinants 
for the cases where the number of vertices $k$ is small.

\begin{lem}\label{lem:small}
In the framework of Setup \ref{set:note}, the following identities hold:
\begin{enumerate}[(1)]
\item For $k=1$, $\det (A(j))=1$ for any vertex $h_j$.
\item For $k=2$, the determinant of the adjacency matrix 
for two distinct vertices $h_{j_1},h_{j_2}$ is given by:
\begin{equation*}
\det (A(j_1, j_2)) =
\begin{cases}
0 & \text{if } \{h_{j_1}, h_{j_2}\} \in E(\Tilde{G}(\HH)), \\
1 & \text{if } \{h_{j_1}, h_{j_2}\} \notin E(\Tilde{G}(\HH)).
\end{cases}
\end{equation*}
\end{enumerate}
\end{lem}

\begin{proof}
\begin{enumerate}
    \item  By the definition of the closure  $\Tilde{G}_+(\HH)$, 
    every vertex $h_j$ has a directed loop. Thus, 
    $A(j)$ is the $1 \times 1$ identity matrix $(1)$, which yields 
    $\det (A(j))=1$.
    \item  If $\{h_{j_1}, h_{j_2}\} \in E(\Tilde{G}(\HH))$,
    by the symmetric construction of the closure, the submatrix 
    $A(j_1,j_2)$ is the $2 \times 2$ all-ones matrix:
    \begin{equation*}
    A(j_1, j_2) = \begin{pmatrix} 1 & 1 \\ 1 & 1 \end{pmatrix}.
    \end{equation*}
    Clearly, its determinant is $0$. 
    Conversely, if $\{h_{j_1}, h_{j_2}\} \notin E(\Tilde{G}(\HH))$, 
    the closure rule assigns exactly one directed edge between them.
    Consequently, $A(j_1,j_2)$ 
    becomes either a lower or upper triangular matrix with unit diagonal entries:
    \begin{equation*}
\begin{pmatrix} 1 & 0 \\ 1 & 1 \end{pmatrix} \quad \text{or} \quad \begin{pmatrix} 1 & 1 \\ 0 & 1 \end{pmatrix}.
\end{equation*}
In either case, we obtain $\det (A(j_1,j_2))=1$.
\end{enumerate}
\end{proof}

Thus, by Lemma \ref{lem:small},
it suffices to discuss the case of $k \ge 3$.

The following lemma reveals 
a structural property of the closure 
when there are no edges among a given set of vertices in $\Tilde{G}(\HH)$.
This property plays a key role in 
simplifying the subsequent determinant calculations.

\begin{lem}\label{lem:source_vertex}
In the framework of Setup \ref{set:note},
let $U=\{h_1,\dots h_k\}$
be a subset of vertices $(k \ge 3)$
satisfying $\{h_r,h_s\} \not\in E(\Tilde{G}(\HH))$ 
for any distinct $h_r,h_s \in U$.
Then, there exists a unique vertex $h_j \in U$
that is directed to all other vertices in $U$,
i.e., $a_{j,\ l}=1$ for all $h_l \in U$.
\end{lem}
\begin{proof}
We first prove the existence by contradiction. 
Suppose no such vertex exists, and let $h_{j_1}\in U$
be a vertex that maximizes the out-degree within the induced subgraph:
$d_{\mathrm{out}}(h_{j_1})=\# \{l \mid h_l\in U, \  a_{j_1,l}=1\}$.
By assumption, $d_{\mathrm{out}}(h_{j_1})< k$,
implying there exists $h_m \in U$ such that $a_{j_1,m}=0$.
Since 
there are no edges between any two vertices in $U$ 
with respect to $\Tilde{G}(\HH)$,
the closure condition ensures that $a_{m,j_1}=1$.

Now, for any vertex $h_l \in U\setminus \{h_{j_1}\}$ pointed to by 
$h_{j_1}$ ($a_{j_1,l}=1$), we 
examine the triple $\{h_{j_1},h_m,h_l\}$ 
which forms an independent set in $\Tilde{G}(\HH)$.
Since $a_{m,j_1}=1$ and $a_{j_1,l}=1$,
the acyclicity guaranteed by the closure condition (2) 
forces the orientation $h_m \to h_l$, hence $a_{m,l}=1$.
This implies $h_m$ is directed to every vertex that $h_{j_1}$
points to, as well as to $h_{j_1}$ itself.
Thus, $d_{\mathrm{out}}(h_m)\ge d_{\mathrm{out}}(h_{j_1})+1$,
which contradicts the maximality of $d_{\mathrm{out}}(h_{j_1})$.
Therefore, a vertex pointing to all others must exist.

To establish uniqueness, suppose there exist two distinct vertices 
$h_j,h_{j'} \in U$ that both point to all other vertices.
This implies $a_{j,j'}=a_{j',j}=1$,
which contradicts the closure rule that exactly one directed edge 
exists between vertices.
Thus, such a vertex is unique.
\end{proof}

\begin{lem}\label{lem:det_noncomm}
In the framework of Setup \ref{set:note},
let $J=\{j_1,\dots,j_k\}$ $(k\ge 3)$
be a subset of indices such that no two vertices in $\Tilde{G}_+(J)$ commute. Then, $\det (A(J))=1$.
\end{lem}

\begin{proof}
We show that $A(J)$ can be permuted into an upper triangular matrix
by induction on the number of vertices $k$.

By Lemma \ref{lem:source_vertex}, the induced subgraph  
$\Tilde{G}_+(J)$ contains a unique vertex $h_{j_1}$ 
that points to all other vertices. By placing $j_1$ 
as the first element of our permutation,
the first row of the adjacency matrix consists entirely of $1$.

Next, we consider the induced subgraph on $J \setminus \{j_1\}$,
which still satisfies the condition that 
no two vertices have an edge between them in $\Tilde{G}(\HH)$. 
If the number of remaining vertices is $3$ or more ($k-1 \ge 3$),
we can repeatedly apply Lemma \ref{lem:source_vertex} 
to continue this inductive permutation process. 
On the other hand, if the number of remaining vertices 
drops below $3$ ($k-1 = 1$ or $2$), 
the assumption $\{h_r, h_s\} \notin E(\Tilde{G}(\HH))$ guarantees,
by the structural construction in the proof of Lemma \ref{lem:small},
that this remaining block also forms an upper triangular matrix
with unit diagonal entries.

Consequently, under this induced ordering, 
the entire adjacency matrix $A(J)$ becomes an upper triangular matrix
with $1$ on all diagonal entries.
Therefore, its determinant is immediately evaluated as $1$.
\end{proof}

To treat row and column operations symmetrically,
we first observe the following lemma.

\begin{lem}\label{lem:trans}
In the framework of Setup \ref{set:note},
let $\Tilde{G}_0$ be a closure of $\Tilde{G}(\HH)$.
Then there exists another closure $\Tilde{G}_0'$
such that
\[
A(\Tilde{G}_0') = {}^tA(\Tilde{G}_0).
\]
\end{lem}

\begin{proof}
The digraph $\Tilde{G}_0'$ is obtained from $\Tilde{G}_0$
by reversing the orientation of every directed edge.
Since the defining conditions of a closure
are invariant under reversing all orientations,
$\Tilde{G}_0'$ is again a closure of $\Tilde{G}(\HH)$.
The adjacency matrix is therefore the transpose of
$A(\Tilde{G}_0)$.
\end{proof}

By Lemma \ref{lem:trans},
the transpose of the adjacency matrix of any closure
is again realized as the adjacency matrix of a closure.
Since taking the transpose preserves determinants,
any argument involving rows can be applied symmetrically to columns.

The following 
lemma highligts the structure of the connected components 
of the noncommutation graph,
which will be implicitly used in the proofs for Cases $(3)$ and $(4)$.

\begin{lem}\label{lem:connected}
In the framework of Setup \ref{set:note},
we uniquely partition the vertex set of the noncommutation graph 
$\Tilde{G}(\HH)$ into its connected components,
$V(\Tilde{G}(\HH)) = \bigsqcup_{i=1}^k V_i$,
where two distinct vertices $u, v$ 
belong to the same component $V_i$
if and only if there exists an undirected path
between them in $\Tilde{G}(\mathcal{H})$. 
Then, for any component $V_i$ and any distinct vertices $v_1, v_2 \in V_i$, 
they are parallel to all external vertices 
$v \in V(\Tilde{G}(\HH)) \setminus V_i$ 
in the closure $\Tilde{G}_+(\HH)$. 
That is, the following identity holds:
\begin{equation*}
a_{v_1,v} = a_{v_2,v} \quad 
\text{for all } v \in V(\Tilde{G}(\mathcal{H})) \setminus V_i.
\end{equation*}
\end{lem}
\begin{proof}
By Definition \ref{def:closure},
the directed edges between different connected components 
in the closure $\Tilde{G}_+(\HH)$
must be acyclic.
Therefore, for any external vertex $v \notin V_i$, 
either there are directed edges from $v$ to both $v_1$ and $v_2$,
or there are directed edges from both $v_1$ and $v_2$ to $v$ in 
$\Tilde{G}_+(\HH)$.
This implies that either 
$a_{v_1,v} = a_{v_2,v} = 1$ or 
$a_{v_1,v} = a_{v_2,v} = 0$ holds,
which yields the desired identity.
\end{proof}

We first consider Case $(3)$,
where the induced subgraph contains an isolated edge
(i.e., a connected component consisting of exactly two vertices).
The following lemma shows that 
the presence of such a component forces 
the determinant to vanish due to 
the immediate emergence of identical rows.

\begin{lem}\label{lem:case3}
In the framework of Setup \ref{set:note}, 
let $J=\{j_1,\dots,j_k\}$ $(k\ge 3)$
be a subset of indices. If the induced subgraph $\Tilde{G}_+(J)$ 
contains an isolated edge $\{h_{j_1},h_{j_2}\}$, then 
the adjacency matrix $A(J)$ has 
identical $j_1$-th and $j_2$-th rows (or columns).
Consequently, we have:
\begin{equation*}
    \det (A(J))=0.
\end{equation*}
\end{lem}
\begin{proof}
    By Lemma \ref{lem:trans}, it suffices to show the case for rows.
    By the definition of an isolated edge in the noncommutation graph,
    any other vertex $h_j\  (j \in J\setminus \{j_1,j_2\} )$ 
    satisfies both $\{h_{j_1},h_j\}\notin E(\Tilde{G}(\HH))$ and 
    $\{h_{j_2},h_j\}\notin E(\Tilde{G}(\HH))$.
    Therefore, by Definition \ref{def:closure}, 
    the adjacency with $h_j$ is identical 
    for both $h_{j_1},h_{j_2}$;
    that is, $a_{j_1,j}=a_{j_2,j}$ 
    for all $j \in J \setminus \{j_1,j_2\}$.

    Furthermore, by the construction of the closure,
    the diagonal entries satisfy $a_{j_1,j_1}=a_{j_2,j_2}=1$.
    Since $\{h_{j_1},h_{j_2}\}$ is an edge in the commutation graph
    $G(\HH)$, the vertices commute, which implies $a_{j_1,j_2}=a_{j_2,j_1}=1$.

    Combining these cases, we obtain $a_{j_1,j}=a_{j_2,j}$ for all $j \in J$.
    Since the $j_1$-th and $j_2$-th rows of $A(J)$ are identical,
    the matrix is singular, and we immediately conclude that $\det (A(J))=0$.
\end{proof}

We now restrict our attention to Case $(4)$,
where the connected components of $\Tilde{G}(\HH)$
have a size of at least three. 
The following lemma guarantees that 
we can inductively extend the common initial segments of parallel rows.

\begin{lem}\label{lem:para}
In the framework of Setup \ref{set:note}, 
let $\Tilde{G}_+(j_1,\dots,j_m)$ be an induced subgraph
with $m\ge3$, and suppose that
\[
\{h_{j_1},h_{j_2}\}\in E(\Tilde{G}(\HH)).
\]

Let $A(j_1,\dots,j_m)=(a_{s,t})_{1 \le s,t \le m}$
be its adjacency matrix.
Assume that $a_{1,l}=a_{2,l}$ $(1\le l\le k)$,
and $a_{1,k+1}\ne a_{2,k+1}$, where $k<m$.
Then there exists a unique $\alpha\in\{1,2\}$
such that $a_{\alpha,k+1}=1$, and
$a_{\alpha,l}=a_{k+1,l}$ $(1\le l\le k+1)$.
\end{lem}

\begin{proof}
Relabel the vertices
$h_{j_1},\dots,h_{j_m}$
as $1,\dots,m$.
Without loss of generality, we assume $a_{1,k+1}=0$, $a_{2,k+1}=1$.
Since $\{1,2\}\in E(\Tilde G(\HH))$,
we have $a_{1,2}=a_{2,1}=1$. Hence $a_{k+1,1}=a_{k+1,2}=1$
by Definition \ref{def:closure}.

Now let $3\le l\le k$.
Since $a_{1,l}=a_{2,l}$, the pair $(1,l)$ and $(2,l)$
falls into one of the following four cases:

\begin{enumerate}[(1)]
\item
$\{1,l\},\{2,l\}\in E(\Tilde G(\HH))$;

\item
$\{1,l\}\in E(\Tilde G(\HH))$ and
$\{2,l\}\notin E(\Tilde G(\HH))$;

\item
$\{1,l\}\notin E(\Tilde G(\HH))$ and
$\{2,l\}\in E(\Tilde G(\HH))$;

\item
$\{1,l\},\{2,l\}\notin E(\Tilde G(\HH))$.
\end{enumerate}

In each case, applying the closure conditions
to the triangles
\[
(1,l,k+1)
\quad\text{and}\quad
(2,l,k+1)
\]
forces
\[
a_{k+1,l}=a_{2,l}.
\]

Therefore,
\[
a_{k+1,l}=a_{2,l}
\qquad
(1\le l\le k+1),
\]
as required.
\end{proof}

Finally, by combining the parallel property
from Lemma \ref{lem:connected} with the inductive extension
from Lemma \ref{lem:para}, 
we can establish that the determinant vanishes 
for any induced subgraph containing at least one edge.

\begin{lem}\label{lem:det0}
In the framework of Setup \ref{set:note}, 
let $\Tilde{G}_+(j_1,\dots,j_m)$ be the induced subgraph with $m\ge 3$.
Let $E(\Tilde{G}(j_1,\dots,j_m)) \neq \emptyset$.
Then the following holds:
\begin{equation*}
    \det (A(j_1,\dots,j_m))=0.
\end{equation*}
\end{lem}
\begin{proof}
We divide the proof into two cases according to whether 
$\Tilde{G}(j_1,\dots,j_m)$ contains an isolated component.

If an isolated component exists, the claim follows 
from Lemma \ref{lem:case3}. Hence, we may assume that 
no isolated component exists. 
Under this assumption, we show that 
the adjacency matrix $A(j_1,\dots,j_m)$ contains
two identical rows or columns. 
By Lemma \ref{lem:trans}, it suffices to consider rows.
By Lemma \ref{lem:connected}, 
we may assume without loss of generality that
$\Tilde{G}(\mathcal{H})$ is a connected graph.

Since $E(\Tilde{G}(j_1,\dots,j_m)) \neq \emptyset$ 
and $\Tilde{G}(\HH)$ is connected,
the induced subgraph contains at least $m-1$ edges
(i.e., $\#E(\Tilde{G}(j_1,\dots,j_m)) \ge m-1$).
We choose one of these edges and relabel the vertices 
so that $\{1,2\} \in E(\Tilde G)$.
If the first and second rows coincide completely,
then $\det(A(j_1,\dots,j_m))=0$. 
Otherwise, let $k+1$ be the first index such that $a_{1,k+1} \neq a_{2,k+1}$,
that is, $a_{1,l}=a_{2,l}$ for all $1 \le l \le k$.

By Lemma \ref{lem:para},
there exist $\alpha \in \{1,2\}$ and an integer $s \ge k+1$
such that $a_{\alpha, l}=a_{k+1,l}$ for all $1 \le l \le s$.
If $s=m$, then the $\alpha$-th row and the $(k+1)$-th row
coincide completely, which yields $\det(A(j_1,\dots,j_m))=0$.

Otherwise, if $s < m$, let $s+1$ be the first index 
where these two rows differ.
Applying Lemma \ref{lem:para} again, we obtain another pair of rows
whose common initial segment has a length of at least $s+1$.
Thus, at each step, the length of the common initial segment 
strictly increases.

Here, we ensure that this extension process never 
terminates before reaching two identical rows.
Suppose for contradiction that the extension of 
the common length stops at some index $s < m$. 
By Lemma \ref{lem:para}, to determine the next entries, 
the row comparison would be forced to rely on the adjacency relation 
of a previously examined edge. This implies that the two rows must 
remain parallel at the $(s_{\mathrm{prev}}+1)$-th position,
where $s_{\mathrm{prev}}$ ($s_{\mathrm{prev}} \le s$) 
is the common length associated with that previous edge.
However, the extension process at that earlier step was terminated
precisely because the two rows were not parallel
at the $(s_{\mathrm{prev}}+1)$-th entry. 
This directly contradicts the fact that the current rows are
already guaranteed to be parallel up to the larger length $s$.
Therefore, such a configuration is impossible,
and the common length must successfully reach $m$.

Since this length is bounded above by $m$,
the process necessarily terminates after at most $m-2$ steps,
successfully yielding two identical rows. 
Therefore, we conclude that
\[
\det(A(j_1,\dots,j_m))=0.
\]
\end{proof}

The following lemma summarizes the evaluation of the determinant 
for all possible configurations of the induced subgraphs,
showing that the value is completely 
determined by the adjacency in the noncommutation graph.

\begin{lem}\label{lem:whole}
In the framework of Setup \ref{set:note},
for any subset of indices $J \subseteq \{1, \dots, n\}$,
the matrix $A(J)$ satisfies:
    \begin{equation*}
        \det (A(J))=
        \begin{cases}
            1 & \text{if} \ E(\Tilde{G}(J))=\emptyset, \\
            0 & \text{otherwise}.
        \end{cases}
    \end{equation*}
\end{lem}
\begin{proof}
The assertion is established by combining the results obtained 
for each cardinality $k=\#J$ and 
the structure of the corresponding edge set.

First, for the cases where k is small $(k=1,2)$,
the claim follows directly from Lemma \ref{lem:small}.
Next, consider the case $k \ge 3$. 
We distinguish the two possibilities for the edge set of the induced subgraph:
\begin{itemize}
    \item If $E(\Tilde{G}(J))=\emptyset$, then 
    $\det (A(J))=1$ by Lemma \ref{lem:det_noncomm}.
    \item If $E(\Tilde{G}(J))\neq \emptyset$, then 
    $\det (A(J))=0$ by Lemma \ref{lem:det0}.
\end{itemize}
Since these cases exhaust all possibilities for the edge configuration,
the proof is complete.
\end{proof}

The following lemma establishes the connection
between the edge set of the induced subgraph and 
the set of cliques $\mathcal{C}(G(\HH))$.

\begin{lem}\label{lem:hikeset}
In the framework of Setup \ref{set:note},
the subset of vertices 
$\{[h_{j_1}],\dots,[h_{j_k}]\}$ forms a clique in $G(\HH)$
(i.e., belongs to $\mathcal{C}(G(\HH))$)
if and only if the edge set of the induced subgraph vanishes:
\begin{equation*}
    E(\Tilde{G}(j_1,\dots,j_k))=\emptyset.
\end{equation*}
\end{lem}
\begin{proof}
By definition, $\{[h_{j_1}],\dots,[h_{j_k}]\} \in \mathcal{C}(G(\HH))$
holds if and only if the simple cycles are pairwise disjoint;
that is, $V(h_{l_1})\cap V(h_{l_2})= \emptyset$ for all 
$j_1 \le l_1<l_2 \le j_k$. According to the construction of the graph 
$\Tilde{G}(\HH)$, this disjointness condition is
equivalent to saying that there is no edge
between any pair of these vertices: 
$\{h_{l_1},h_{l_2}\} \notin E(\Tilde{G}(\HH))$ for all 
$j_1 \le l_1<l_2 \le j_k$. This immediately implies that the induced subgraph 
contains no edges, i.e., $E(\Tilde{G}(j_1,\dots,j_k))=\emptyset$.
\end{proof}

The following is one of our main theorems.

\begin{thm}\label{thm:cliquehike}
    In the framework of Setup \ref{set:note}, 
    let $p_{\mathcal{H}}(G(\mathcal{H}))$ be the clique polynomial on $\ZZ \langle \mathcal{H} \rangle$.
    Then the following holds:
    \begin{equation*}
        \mathrm{Cdet}(I_n-D(\Tilde{G}_+(\HH))A(\Tilde{G}_+(\HH))
        =p_{\HH}(G(\HH)).
    \end{equation*}
\end{thm}
\begin{proof}
By the equation \eqref{eq:smalldet_form} in Lemma \ref{lem:smalldet},
the Cayley determinant on the left-hand side can be written as:
\begin{equation}\label{eq:Cdet_reduced}
\sum_{k=0}^n (-1)^k \sum_{1 \le j_1 < \dots < j_k \le n} 
(h_{j_1} h_{j_2} \cdots h_{j_k}) \cdot \det (A(j_1,\dots,j_k)).
\end{equation}

By Lemma \ref{lem:whole}, 
the determinant $\det (A(j_1,\dots,j_k))$ equals $1$ if 
$E(\Tilde{G}(j_1,\dots,j_k)) = \emptyset$,
and $0$ otherwise. Thus, the summation \eqref{eq:Cdet_reduced} reduces to:
\begin{equation*}
\sum_{k=0}^n (-1)^k \sum_{\substack{1 \le j_1 < \dots < j_k \le n \\ E(\Tilde{G}(j_1,\dots,j_k)) = \emptyset}} h_{j_1} h_{j_2} \cdots h_{j_k}.
\end{equation*}

Furthermore, Lemma \ref{lem:hikeset} states that 
the condition $E(\Tilde{G}(j_1,\dots,j_k)) = \emptyset$
holds if and only if the corresponding subset of vertices
$\{[h_{j_1}],\dots,[h_{j_k}]\}$ forms a clique in the commutation graph $G(\HH)$.

Therefore, by rewriting the summation directly over the set of cliques 
$\mathcal{C}(G(\HH))$, we obtain:
\begin{equation*}
\sum_{c \in \mathcal{C}(G(\mathcal{H}))} (-1)^{\#c} \prod_{[h_j] \in c} h_j,
\end{equation*}
where $\#c$ denotes the number of vertices in the clique $c$. 
By Definition \ref{def:cliquepolyonM}, 
this precisely matches the definition of the clique polynomial
$p_{\HH}(G(\HH))$ on $\ZZ \langle \HH \rangle$. This completes the proof.
\end{proof}

By applying the ring homomorphism $\psi_1$ 
to the identity established in Theorem \ref{thm:cliquehike},
we obtain the following result on the polynomial ring $\ZZ [u]$.

\begin{cor}\label{cor:hikesum}
In the framework of Setup \ref{set:note}, 
let $D_w(\Tilde{G}_+(\HH))=\psi_1(D(\Tilde{G}_+(\HH)))$
be the image of the weight matrix under $\psi_1$, given by 
$\mathrm{diag}[u^{l(h_1)},\dots,u^{l(h_n)}]$.
Let $p_w(G(\HH),u)$ be the vertex-weighted clique polynomial on $\ZZ[u]$.
    Then the following holds:
    \begin{equation*}
        \det(I_n-D_w(\Tilde{G}_+(\HH))A(\Tilde{G}_+(\HH))=p_w(G(\mathcal{H}),u), 
    \end{equation*}
    where the sum runs over all cliques in $G(\HH)$.
\end{cor}
\begin{proof}
By applying the ring homomorphism $\psi_1:\ZZ \langle \HH \rangle \to \ZZ[u]$
to both sides of the identity in Theorem \ref{thm:cliquehike},
the assertion follows immediately.

Specifically, on the left-hand side, since $\psi_1$
maps the matrix $D(\Tilde{G}_+(\HH))$ to $D_w(\Tilde{G}_+(\HH))$
entrywise and maps the commutative determinant $\mathrm{Cdet}$
to the standard matrix determinant $\det$ over $\ZZ[u]$, we have:
\begin{equation*}
\psi_1 \left( \mathrm{Cdet}(I_n - D(\tilde{G}_+(\mathcal{H}))
A(\tilde{G}_+(\mathcal{H}))) \right) 
= \det(I_n - D_w(\tilde{G}_+(\mathcal{H}))A(\tilde{G}_+(\mathcal{H}))).
\end{equation*}

On the right-hand side, by Remark \ref{rem:12(3)} (\ref{rem:(3)}),
the image of the clique polynomial $p_{\HH}(G(\HH))$
under $\psi_1$ yields 
the vertex-weighted clique polynomial. 
Combining these two relations completes the proof.
\end{proof}

The following theorem successfully establishes our primary goal:
showing that the determinant of the closure adjacency matrix
explicitly yields the reciprocal of the Ihara zeta function 
for the digraph $G$.

\begin{thm}[Theorem \ref{thm:main1}]
In the framework of Setup \ref{set:note},
let $D(\Tilde{G}_+(\HH))$ be the image of 
the diagonal matrix under $\psi_1$, 
given by $\mathrm{diag}[u^{l(h_1)},\dots,u^{l(h_n)}]$.
Then, the reciprocal of the Ihara zeta function of the digraph $G$
is given by the following determinant formula:
    \begin{equation*}
        \zeta(G,u)^{-1} = \det(I_n - D_w(\Tilde{G}_+(\HH))A(\Tilde{G}_+(\HH))).
    \end{equation*}
\end{thm}
\begin{proof}
By Proposition \ref{prop:hashimoto}, Proposition \ref{prop:ponst},
and Lemma \ref{lem:hashiwei}, the reciprocal of the Ihara zeta function
$\zeta(G,u)^{-1}$ of the digraph $G$ is
identically equal to the vertex-weighted clique polynomial $p_w(G(\HH),u)$
on $\ZZ[u]$.

On the other hand, Corollary \ref{cor:hikesum} guarantees that 
this vertex-weighted clique polynomial coincides precisely 
with the determinant on the right-hand side of the target identity:
\begin{equation*}
p_w(G(\mathcal{H}),u) = \det(I_n - D_w(\tilde{G}_+(\mathcal{H}))A(\tilde{G}_+(\mathcal{H}))).
\end{equation*}
Combining these two relations immediately 
yields the desired identity, which completes the proof.
\end{proof}

To illustrate our main results,
we provide a concrete example demonstrating how 
the determinant of the closure adjacency matrix explicitly
yields the reciprocal of the Ihara zeta function,
and show its equivalence to the 
classical Hashimoto-type determinant expression.

\begin{ex}

Consider a digraph $G$
and its corresponding line digraph without backtrack $L_{WB}(G)$,
as illustrated in Figure \ref{fig:graphs} and Figure \ref{fig:linedigraphs}.

\begin{figure}[ht]
    \centering
    \begin{minipage}{0.45\textwidth}
\begin{tikzpicture}[
    >=Stealth, 
    every node/.style={circle, draw=black, fill=black, inner sep=0pt, minimum size=4pt},
    main edge/.style={->, shorten >=1pt, shorten <=1pt, thick}, 
    edge label/.style={circle, fill=white, draw=black, inner sep=1.5pt, font=\bfseries\small, text=black}
]

\node (v6) at (-1.5,  2.0) {}; 
\node (v1) at ( 1.5,  2.0) {}; 

\node (v4) at (-1.5, -2.0) {}; 
\node (v3) at ( 1.5, -2.0) {}; 

\node (v5) at (-3.0,  0.0) {}; 
\node (v2) at ( 3.0,  0.0) {}; 


\draw[main edge, bend left=25]  (v6) to node[edge label, midway] {6} (v1);
\draw[main edge, bend left=25]  (v1) to node[edge label, midway] {7} (v6);

\draw[main edge, bend left=25]  (v4) to node[edge label, midway] {9} (v3);
\draw[main edge, bend left=25]  (v3) to node[edge label, midway] {3} (v4);

\draw[main edge] (v6) -- node[edge label, midway] {8} (v4);
\draw[main edge] (v3) -- node[edge label, midway] {10} (v1);

\draw[main edge] (v1) -- node[edge label, midway] {1} (v2);
\draw[main edge] (v2) -- node[edge label, midway] {2} (v3);

\draw[main edge] (v4) -- node[edge label, midway] {4} (v5);
\draw[main edge] (v5) -- node[edge label, midway] {5} (v6);

\end{tikzpicture}
    \caption{Digraph $G$}
    \label{fig:graphs}
\end{minipage}
\hfill
    \begin{minipage}{0.52\textwidth}
    \centering
\begin{tikzpicture}[
    >=Stealth,
    vertex/.style={draw=black, circle, fill=white, inner sep=2pt, minimum size=14pt, font=\bfseries\small, thick},
    main edge/.style={->, shorten >=1pt, shorten <=1pt, thick, color=black}, 
    scale=0.7
]


\pgfmathsetmacro{\R}{3.5}
\node[vertex] (v1) at (120:\R) {1};
\node[vertex] (v2) at (60:\R)  {2};
\node[vertex] (v3) at (0:\R)   {3};
\node[vertex] (v4) at (-60:\R) {4};
\node[vertex] (v5) at (-120:\R){5}; 
\node[vertex] (v6) at (180:\R) {6};

\pgfmathsetmacro{\r}{1.5}
\node[vertex] (v10) at (90:\r)  {10};
\node[vertex] (v7)  at (0:\r)   {7}; 
\node[vertex] (v8)  at (-90:\r) {8};
\node[vertex] (v9)  at (180:\r) {9};


\draw[main edge] (v1) -- (v2);
\draw[main edge] (v2) -- (v3);
\draw[main edge] (v3) -- (v4);
\draw[main edge] (v4) -- (v5);
\draw[main edge] (v5) -- (v6);
\draw[main edge] (v6) -- (v1);

\draw[main edge] (v9)  -- (v10);
\draw[main edge] (v10) -- (v7);
\draw[main edge] (v7)  -- (v8);
\draw[main edge] (v8)  -- (v9);

\draw[main edge] (v10) -- (v1);  
\draw[main edge] (v2)  -- (v10); 
\draw[main edge] (v5)  -- (v8);  
\draw[main edge] (v8)  -- (v4);  

\end{tikzpicture}
\caption{$L_{WB}(G)$ of $G$}
\label{fig:linedigraphs}
    \end{minipage}
\end{figure}

For convenience, 
we label each directed edge of $G$
(which corresponds to a vertex in $L_{WB}(G)$
as shown above. In the graph $L_{WB}(G)$,
we can identify exactly six simple cycles.
By listing their vertices in the order of traversal,
these cycles are given by:
\begin{enumerate}[$h_1$:]
    \item 1,2,3,4,5,6
    \item 1,2,10
    \item 4,5,8
    \item 7,8,9,10
    \item 1,2,10,7,8,4,5,6
    \item 1,2,3,4,5,8,9,10
\end{enumerate}

Based on the commutation relations among these cycles,
the noncommutation graph $\tilde{G}(\HH)$ and one of its closures, 
$\tilde{G}_+(\HH)$, are constructed as shown 
in Figure \ref{fig:noncomu} and Figure \ref{fig:closures}. 
For enhanced visibility, the vertices in these graphs are labeled with the index $i$ instead of $h_i$.

\begin{figure}
    \centering
    \begin{minipage}{0.48\textwidth}
        \begin{tikzpicture}[
    thick,
    auto, 
    every edge/.style={draw, thick, shorten >=1.5pt, shorten <=1.5pt},
    every node/.style={draw, thick, circle, inner sep=1pt, minimum size=0.6cm, font=\small\bfseries},
    scale=0.9, transform shape]

    \node (1) at (90:2.5cm) {1};
    \node (2) at (30:2.5cm) {2};
    \node (3) at (-30:2.5cm) {3};
    \node (4) at (-90:2.5cm) {4};
    \node (5) at (-150:2.5cm) {5};
    \node (6) at (150:2.5cm) {6};

    
    \draw (1) to (2);
    \draw (1) to (3);
    \draw (1) to (5);
    \draw (1) to (6);

    \draw (2) to (4);
    \draw (2) to (5);
    \draw (2) to (6);

    \draw (3) to (4);
    \draw (3) to (5);
    \draw (3) to (6);

    \draw (4) to (5);
    \draw (4) to (6);

    \draw (5) to (6);

    \draw [thick, out=60,   in=120,  min distance=1cm] (1) to (1);
    \draw [thick, out=0,    in=60,   min distance=1cm] (2) to (2);
    \draw [thick, out=-60,  in=0,    min distance=1cm] (3) to (3);
    \draw [thick, out=-120, in=-60,  min distance=1cm] (4) to (4);
    \draw [thick, out=-180, in=-120, min distance=1cm] (5) to (5);
    \draw [thick, out=120,  in=180,  min distance=1cm] (6) to (6);
    
\end{tikzpicture}
\caption{The noncommutation graph $\tilde{G}(\HH)$}
\label{fig:noncomu}
    \end{minipage}
\begin{minipage}{0.48\textwidth}
    \begin{tikzpicture}[
    thick, 
    ->, >=stealth, 
    auto, 
    every edge/.style={draw, thick, >=stealth, shorten >=1.5pt, shorten <=1.5pt},
    every node/.style={draw, thick, circle, inner sep=1pt, minimum size=0.6cm, font=\small\bfseries},
    scale=0.9, transform shape]

    \node (1) at (90:2.5cm) {1};
    \node (2) at (30:2.5cm) {2};
    \node (3) at (-30:2.5cm) {3};
    \node (4) at (-90:2.5cm) {4};
    \node (5) at (-150:2.5cm) {5};
    \node (6) at (150:2.5cm) {6};

    
    \draw (1) to (2);
    \draw [bend right=12] (1) to (3);
    \draw [bend right=18] (1) to (4);
    \draw [bend right=12] (1) to (5);
    \draw (1) to (6);

    \draw [bend left=12] (2) to (1); 
    \draw (2) to (3);
    \draw [bend right=18] (2) to (4);
    \draw [bend left=12] (2) to (5);
    \draw [bend right=18] (2) to (6);

    \draw [bend right=12] (3) to (1); 
    \draw (3) to (4);
    \draw [bend left=18] (3) to (5);
    \draw [bend left=12] (3) to (6);

    \draw [bend right=18] (4) to (2); 
    \draw [bend right=12] (4) to (3);
    \draw (4) to (5);
    \draw [bend left=12] (4) to (6);

    \draw (5) to (1);
    \draw [bend right=18] (5) to (2);
    \draw [bend right=12] (5) to (3);
    \draw [bend right=12] (5) to (4);
    \draw (5) to (6);

    \draw [bend left=18] (6) to (1);
    \draw [bend left=18] (6) to (2);
    \draw [bend right=18] (6) to (3);
    \draw [bend right=18] (6) to (4);
    \draw [bend right=18] (6) to (5);

    
    \draw [->, thick, out=60, in=120, min distance=1cm] (1) to (1);
    
    \draw [->, thick, out=0, in=60, min distance=1cm] (2) to (2);
    
    \draw [->, thick, out=-60, in=0, min distance=1cm] (3) to (3);
    
    \draw [->, thick, out=-120, in=-60, min distance=1cm] (4) to (4);
    
    \draw [->, thick, out=-180, in=-120, min distance=1cm] (5) to (5);
    
    \draw [->, thick, out=120, in=180, min distance=1cm] (6) to (6);
    
\end{tikzpicture}
\caption{A closure $\tilde{G}_+(\HH)$}
\label{fig:closures}
\end{minipage}

\end{figure}

Since the closure operation is generally not unique, 
we explicitly specify our choice 
in constructing $\tilde{G}_+(\mathcal{H})$ 
from $\tilde{G}(\mathcal{H})$.
For the non-adjacent pairs $\{1, 4\}$ and $\{2, 3\}$
in $\tilde{G}(\mathcal{H})$,
which correspond to commuting cycles,
we choose to insert the directed edges as follows:
\begin{align*}
    1 \to 4 \quad &\text{and} \quad 4 \not\to 1, \\
    2 \to 3 \quad &\text{and} \quad 3 \not\to 2.
\end{align*}
By combining these choices with the closure rule
for all other pairs of vertices,
we obtain the directed graph $\tilde{G}_+(\mathcal{H})$.

The adjacency matrix $A(\tilde{G}_+(\HH))$ 
of the closure graph is explicitly given by:
\begin{equation*}
A(\tilde{G}_+(\HH))=
    \begin{pmatrix}
        1 & 1 & 1 & 1 & 1 & 1 \\
        1 & 1 & 1 & 1 & 1 & 1 \\
        1 & 0 & 1 & 1 & 1 & 1 \\
        0 & 1 & 1 & 1 & 1 & 1 \\
        1 & 1 & 1 & 1 & 1 & 1 \\
        1 & 1 & 1 & 1 & 1 & 1
    \end{pmatrix}
\end{equation*}

Applying the main determinant formula established in Theorem \ref{thm:main1},
we compute $\det(I_n-D_w(\Tilde{G}_+)A(\Tilde{G}_+))$.
By substituting the respective cycle lengths to obtain the diagonal entries 
$u^{l(h_i)}$, this determinant calculation yields:
\begin{equation*}
    \det \left( 
    I_6-
    \begin{pmatrix}
        u^6 &&&&& \\
        & u^3 &&&&\\
        &&u^3 &&&\\
        &&&u^4 &&\\
        &&&&u^8 & \\
        &&&&& u^8
    \end{pmatrix} \cdot 
     \begin{pmatrix}
        1 & 1 & 1 & 1 & 1 & 1 \\
        1 & 1 & 1 & 1 & 1 & 1 \\
        1 & 0 & 1 & 1 & 1 & 1 \\
        0 & 1 & 1 & 1 & 1 & 1 \\
        1 & 1 & 1 & 1 & 1 & 1 \\
        1 & 1 & 1 & 1 & 1 & 1
    \end{pmatrix}
    \right)
    =u^{10} - 2u^{8} - u^{4} - 2u^{3} + 1.
\end{equation*}

To verify this result,
we compute the reciprocal of the Ihara zeta function
via the classical Hashimoto-type expression $\det(I_{10}-A(L_{WB}(G))u)$,
where $A(L_{WB}(G))$ is the adjacency matrix of the line digraph without
backtrack $L_{WB}(G)$.
The corresponding $10\times 10$ determinant is given by:
\begin{equation*}
    \det 
    \begin{pmatrix}
        1 & -u & 0 & 0 & 0 & 0 & 0 & 0 & 0 & 0 \\
        0 & 1 & -u & 0 & 0 & 0 & 0 & 0 & 0 & -u \\
        0 & 0 & 1 & -u & 0 & 0 & 0 & 0 & 0 & 0 \\
        0 & 0 & 0 & 1 & -u & 0 & 0 & 0 & 0 & 0 \\
        0 & 0 & 0 & 0 & 1 & -u & 0 & -u & 0 & 0 \\
        -u & 0 & 0 & 0 & 0 & 1 & 0 & 0 & 0 & 0 \\
        0 & 0 & 0 & 0 & 0 & 0 & 1 & -u & 0 & 0 \\
        0 & 0 & 0 & -u & 0 & 0 & 0 & 1 & -u & 0 \\
        0 & 0 & 0 & 0 & 0 & 0 & 0 & 0 & 1 & -u \\
        -u & 0 & 0 & 0 & 0 & 0 & -u & 0 & 0 & 1 \\
    \end{pmatrix}
    =u^{10} - 2u^{8} - u^{4} - 2u^{3} + 1
\end{equation*}

Both approaches yield the exact same polynomial,
confirming the validity of our main theorem. 
\end{ex}

\section*{Acknowledgments}
The author would like to thank Professor Hidekazu Furusho for his
helpful comments and advices.
The author is grateful to Professor Iwao Sato for his helpful comments 
on results and during revisions of this paper.
The author is deeply indebted to Professor Yoshinori Yamasaki for his insightful suggestions and valuable comments in revising this paper.

\end{document}